\documentclass[12pt]{amsart}
\usepackage[left=1in,top=1in,right=1in,bottom=1in,letterpaper]{geometry}
\usepackage[all]{xy}
\usepackage{epsfig}
\usepackage{amsmath}
\usepackage{amssymb}
\usepackage{amscd}
\usepackage{bm}
\usepackage{bbm}

\usepackage{url}
\usepackage{verbatim} 
\usepackage{color}
\usepackage{epsfig}
\usepackage{stmaryrd}
\usepackage{amsthm}
\usepackage{pifont}
\usepackage{mathrsfs}
\usepackage{wasysym}
\usepackage{tikz-cd}
\usetikzlibrary{decorations.pathmorphing}
\usepackage{graphicx}
\usepackage{xcolor}
\definecolor{dark-red}{rgb}{0.4,0.15,0.15}
\definecolor{dark-blue}{rgb}{0.15,0.15,0.4}
\definecolor{medium-blue}{rgb}{0,0,0.5}
\usepackage{xr}
\usepackage{hyperref}
\hypersetup{
    colorlinks, linkcolor={dark-red},
    citecolor={dark-blue}, urlcolor={medium-blue}
}

\title[]{Immersions and the space of all translation structures}
\author[]{W. Patrick Hooper}
\thanks{Support was provided by N.S.F. Grant DMS-1101233 and a PSC-CUNY Award (funded by The Professional Staff Congress and The City University of New York).}
\address{
The City College of New York\\
New York, NY, USA 10031}
\email{whooper@ccny.cuny.edu}
\date{\today}

\ifdefined\theorem
\else
\newtheorem{theorem}{Theorem}
\newtheorem{proposition}[theorem]{Proposition}
\newtheorem{lemma}[theorem]{Lemma}
\newtheorem{remark}[theorem]{Remark}
\newtheorem{corollary}[theorem]{Corollary}

\theoremstyle{definition}

\fi
\ifdefined\openquestion
\else
\newtheorem{openquestion}[theorem]{Open Question}
\fi

\ifdefined\question
\else
\newtheorem{question}[theorem]{Question}
\fi

\newlength{\savearraycolsep}
	{\setlength{\savearraycolsep}{\arraycolsep}%
	\setlength{\arraycolsep}{#1}%
	\begin{array}{#2}}%
	{\end{array}\setlength{\arraycolsep}{\savearraycolsep}}

\newcommand{\red}[1]{{\textcolor{red}{#1}}}

\def\N{\mathbb{N}}%
\def\Q{\mathbb{Q}}%
\def\R{\mathbb{R}}%
\def\Z{\mathbb{Z}}%
\def\GL{\textit{GL}}

\def\0{{\mathbf{0}}}
\def\1{{\mathbf{1}}}

\def\u{{\mathbf{u}}}
\def\bu{{\mathbf{u}}} %
\def\bv{{\mathbf{v}}} %
\def\v{{\mathbf{v}}} %
\def\bw{{\mathbf{w}}}

\def\sA{{\mathcal{A}}}
\def\sB{{\mathcal{B}}}

\def\sE{{\mathcal{E}}}

\def\sI{{\mathcal{I}}}
\def\sJ{{\mathcal{J}}}
\def\sK{{\mathcal{K}}}

\def\sM{{\mathcal{M}}}

\def\sO{{\mathcal{O}}}
\def\sP{{\mathcal{P}}}
\def\sR{{\mathcal{R}}}
\def\sS{{\mathcal{S}}}

\def\sU{{\mathcal{U}}}
\def\sV{{\mathcal{V}}}

\def\sX{{\mathcal{X}}}
\def\sY{{\mathcal{Y}}}

\makeatletter
\def\imod#1{\allowbreak\mkern10mu({\operator@font mod}\,\,#1)}
\makeatother

\def\and{{\quad \textrm{and} \quad}}

\newtheorem{convention}[theorem]{Convention}

\newcommand{\PC}{{\textrm{PC}}}
\newcommand{\SC}{{\textrm{SC}}}
\newcommand{\disk}{{\textrm{Disk}}}
\newcommand{\cdisk}{{\overline{\disk}}}

\newcommand{\BC}{{\textrm{BC}}}
\newcommand{\ER}{{\textit{ER}}}
\newcommand{\BE}{{\textit{BE}}}

\newcommand{\tM}{{\tilde \sM}}
\newcommand{\tE}{{\tilde \sE}}

\newcommand{\tpi}{{\tilde \pi}}

\def\emb{\hookrightarrow}
\def\imm{\rightsquigarrow}
\def\dev{{\textit{Dev}}} 
\def\fuse{\curlyvee}

\newif\ifdraft\drafttrue
\draftfalse

\newcommand{\name}[1]{\label{#1}{\ifdraft{\textcolor{blue}{\{\textrm{#1}\}}}\else\ignorespaces\fi}}

\definecolor{darkgreen}{rgb}{0,0.8,0}
\newcommand{\com}[1]{\ifdraft{\textcolor{darkgreen}{#1}}\else\ignorespaces\fi}

\begin{document}
\begin{abstract}
A translation structure on a surface is an atlas of charts to the plane so that the transition functions are translations. We allow our surfaces to be non-compact and infinite genus. We endow the space of all pointed surfaces equipped with a translation structure with a topology, which we call the immersive topology because it is related to the manner in which disks can be immersed into such a surface. We prove that a number of operations typically done to translation surfaces are continuous with respect to the topology. We show that the 
topology is Hausdorff, and that the collection of surfaces with a fixed lower bound on the injectivity radius at the basepoint is compact.
\end{abstract}
\maketitle

\section{Introduction}
A {\em translation structure} on a surface is an atlas of charts to the plane where the transition functions are translations. There is a natural notion of when two such structures are isomorphic (as we explain in \S \ref{sect:translation structures}), and we use the term {\em translation surface} to mean an isomorphism
class of translation structures. A {\em pointed translation surface} is a connected 
translation surface together with a choice of a basepoint.

Translation surfaces are relatively simple geometric objects. Despite their simplicity, there are natural geometric and dynamical questions about these surfaces whose difficulty varies greatly based on the surface chosen. 
	
The primary goal of this article is to topologize the collection of {\em all} pointed translation surfaces.
That is, the space will include translation surfaces of any topological type admitting a translation structure. The topology is designed to allow 
sequences of infinite genus surfaces to converge, and to allow finite genus surfaces to limit on surfaces of infinite genus. The topology also makes it easy to pass geometric information back and forth from the limiting surface to the approximates.

As motivation, we note that translation surfaces come in great variety, and highlight a few examples of interest:
\begin{itemize}
\item The only closed surfaces admitting translation structures are tori. 
\item Oriented surfaces, $S_{g,n}$, of genus $g \geq 2$ with $n \geq 1$ punctures always admit translation structures
whose metric completions are homeomorphic to closed surfaces of genus $g$. Such metric completions introduce cone singularities with cone angles in $2 \pi \Z$. 
\item A natural unfolding construction associates a translation surface to every polygonal billiard table \cite{ZK}.
If the polygon is {\em irrational} (has angles which are irrational multiples of $\pi$), then the translation surface has infinite area and infinite genus.
\end{itemize}
These examples are listed in order of our understanding. The space of translation structures on tori is classically identified with the space of lattices in the plane, $\GL(2,\R)/\GL(2,\Z)$. 
The completions of genus $g \geq 2$ surfaces are traditionally just called translation surfaces, but we will call them {\em finite genus translation surfaces}.
The moduli space of such translation surfaces of genus $g$ can be identified with a vector bundle (with the 
zero section removed) over the moduli spaces of Riemann surfaces of genus $g$. 
These surfaces have been actively studied for the past 35 years.
We refer the reader to the survey articles \cite{MT} and \cite{Zorich06} for an introduction to the subject.
At the other extreme, billiards in irrational polygons are not well understood at all. Basic questions such as the existence of periodic billiard orbits remain unanswered. (Existence of periodic billiard paths is equivalent
to the existence of closed geodesics in the associated translation surface.) We refer the reader to \cite{T05} for background on polygonal billiards.

The space of all translation surfaces is attractive because it contains all the surfaces mentioned
in the above. In particular, it is natural to wonder which surfaces in this space can be best understood (from various viewpoints). 
Perhaps a reasonable expectation is that an infinite genus limit of the best understood finite genus translation surfaces
should be easier to understand than some high genus translation surfaces. 

Indeed, work from the past few years on the geometry and dynamics of infinite genus translation surfaces seems to suggest that there are many such surfaces which can be well understood. There has been widespread recent interest
in the study of abelian covers of closed finite genus translation surfaces. Many such articles studied covers of {\em lattice surfaces} (maximally symmetric finite genus translation surfaces in the sense of Veech \cite{V}):
\cite{HS09}, \cite{HHW10}, \cite{Troubetzkoy10}, \cite{DHL11}, \cite{HLT11}, \cite{Schmoll11preprint}, \cite{HW13}, and \cite{Delecroix13}.
A few papers have studied $\Z$-covers of less symmetric finite genus translation surfaces including
\cite{HW10}, \cite{FUpreprint} and \cite{FUstrip}, \cite{RTarxiv11}, and \cite{RTarxiv12}.
A couple of articles have studied surfaces arising from limiting procedures such as \cite{Bowman13} and \cite{Higl1}.
And a number of articles have studied surfaces with no apparent connection to the finite genus case including
\cite{PSV11}, \cite{BV13}, \cite{TrevinoFinite} and \cite{Hinf}.

Aside from the perceived practical benefits for defining the topology, it is hoped that by formally placing all translation surfaces in the same space, we will provide a context for better understanding progress in the subject.

We briefly discuss the layout of this paper. In \S \ref{sect:translation structures},
we give a set-theoretic description of the moduli space of all translation structures. We also introduce the canonical bundle of translation surfaces over this moduli space. In \S \ref{sect:immersions}, we introduce the idea of immersing one subset of a pointed translation surface into another pointed translation surface. We use this idea to place a topology on the space of pointed translation surfaces. In \S \ref{sect:results}, we describe our main results for this topology. We are interested in understanding this topological space and showing that a number of natural operations done on pointed translation surfaces are continuous. In \S \ref{sect:approach}, we explain our approach to the proofs of the main results. We will utilize our understanding of translation structures on the disk developed in \cite{HooperImmersions1}. In \S \ref{sect:outline}, we outline the remainder of the paper and describe where the main results are proved.

\section{Translation structures}
\name{sect:translation structures}
\subsection{Definition of translation structure}
We take a view of translation structures following the more general idea of a $(G,X)$-structure. See \cite{Thurston}
for an introduction to these ideas.

In this paper, we use the term {\em surface} to mean a connected oriented $2$-manifold $X$ without boundary (which may or may not be closed).
An {\em atlas of charts} from a surface $X$ to $\R^2$ is a collection of orientation preserving 
local homeomorphisms from open subsets of
$X$ to $\R^2$ so that the collection of domains
of these maps cover $X$.
We denote the choice of an atlas by a set of pairs consisting
of the domain and the map, $\sA=\{(U_j,\phi_j)\}$.
Such an atlas $\sA$ is a {\em translation atlas} if for every choice of $i$ and $j$ so that $U_i \cap U_j \neq \emptyset$, the associated {\em transition function},
$$\phi_i \circ \phi_j^{-1}|_{\phi_j(U_i \cap U_j)}~:~\phi_j(U_i \cap U_j) \to \R^2$$
is locally a restriction of a translation. 
A translation atlas on $X$ is {\em maximal} if it is not properly contained in any translation atlas on $X$.
A {\em translation structure} is a pair $(X,\sA)$, where $X$ is a surface and $\sA$ is a maximal translation atlas on $X$. 

Because transition functions are translations, a translation atlas determines a metric on $X$ via pullback.
It also determines a canonical trivialization of the tangent bundle of $X$, $TX \cong X \times \R^2$, by pulling back the standard trivialization of the unit tangent bundle of $\R^2$, $T \R^2 \cong \R^2 \times \R^2$.

\subsection{Translation isomorphisms and moduli space}
\name{sect:notions}
Let $(X,\sA)$ and $(Y,\sB)$ be translation structures, and let 
$\sA=\{(U_i,\phi_i)\}$ and $\sB=\{(V_j,\psi_j)\}$.
A homeomorphism $h:X \to Y$ 
is a {\em translation isomorphism} from the structure $\sA$ to the structure $\sB$ 
if 
$$\sB=\Big\{\big(h(U_i),~\phi_i \circ h^{-1}:h(U_i) \to \R^2\big)\Big\}.$$
That is, we require that the push forward of the atlas $\sA$ under $h$ is the same as the atlas $\sB$.
If such an isomorphism exists we call $(X,\sA)$ and $(Y,\sB)$ {\em translation isomorphic}. 

In order to put a topology on the space of translation structures, we will need to add a basepoint.
A {\em pointed surface} is a pair $(X,x_0)$ consisting of a connected surface $X$ and a {\em basepoint} $x_0 \in X$. 
A {\em pointed translation structure} is a triple $(X,x_0,\sA)$ where $(X,x_0)$ is a pointed surface and $(X,\sA)$ is a translation structure. An isomorphism between the pointed translation structures $(X,x_0,\sA)$ and $(Y,y_0,\sB)$ 
is a translation isomorphism $h:X \to Y$ which respects the basepoints: $h(x_0)=y_0$. If such an isomorphism exists, we call the structures {\em isomorphic}. We write $[X,x_0,\sA]$ to denote the isomorphism class of the pointed translation structure $(X,x_0,\sA)$. 

The {\em moduli space $\sM$ of all (pointed) translation structures} is the collection of all isomorphism classes
of pointed translation structures. This moduli space is naturally partitioned according to the homeomorphism type of the underlying surface. Given $(X,x_0)$, we let $\sM(X,x_0)$ denote the collection of all isomorphism classes of the form $[X,x_0,\sA]$. Then we have
$$\sM=\bigcup_{(X,x_0) \in \sS} \sM(X,x_0),$$
where $\sS$ is a collection of pointed surfaces
with one representative from each homeomorphism class of surfaces.

\subsection{Bundle structure}
\name{sect:bundle}
We will describe a canonical (set-theoretic) surface bundle over $\sM$
so that the fiber over a point in $\sM(X,x_0)$ is isomorphic to $X$. 

Consider the collection of all pointed translation surfaces with an additional point selected,
$$E=\{(X,x_0,\sA;x)~:~\text{$(X,x_0,\sA)$ is a translation structure and $x \in X$}\}.$$
A homeomorphism $h:X \to X$ is a {\em two pointed isomorphism} from $(X,x_0,\sA;x) \in E$ to 
$(Y,y_0,\sB;y) \in E$ if it is an isomorphism from $(X,x_0,\sA)$ to $(Y,y_0,\sB)$ and satisfies
$h(x)=y$. We write $[X,x_0,\sA;x]$ to denote the two pointed isomorphism class of $(X,x_0,\sA;x)$.
We define the {\em total space $\sE$ of translation surfaces} to be the collection of all 
two pointed isomorphism classes.

Note that there is a canonical projection 
$$\pi:\sE \to \sM; \quad [X,x_0,\sA;x] \mapsto [X,x_0,\sA].$$
Moreover, the fibers of this projection are endowed with a natural translation surface structure.
To see this, choose an $[X,x_0,\sA] \in \sM$ and a translation structure $(X,x_0,\sA)$ from the isomorphism
class $[X,x_0,\sA]$. 
Let $\sA=\{(U_i,\phi_i)\}$, and let
$S=\pi^{-1}([X,x_0,\sA]) \subset \sE$ be the fiber. Observe that associated to 
choice of $(X,x_0,\sA) \in [X,x_0,\sA]$, we have a canonical bijection
$$\gamma:X \to S; \quad x \mapsto [X,x_0,\sA;x].$$
We can use $\gamma$ to push the topology from $X$ onto $S$. This makes $S$ a surface. Furthermore,
this topology is independent of our choice. In addition, we get a translation structure on $S$ from the atlas $$\big\{\big(\gamma(U_i), \phi_i \circ \gamma^{-1}\big)\}.$$
Again, this atlas is independent of the choice of translation structure from the isomorphism class $[X,x_0,\sA;x]$.
In summary, we have endowed each fiber $S=\pi^{-1}([X,x_0,\sA])$ 
with a translation structure which is isomorphic to each of the translation structures in the equivalence class $[X,x_0,\sA;x]$.

\begin{convention}
\name{conv1}
A {\em translation surface} is a fiber of the projection $\pi:\sE \to \sM$
endowed with the topology of a surface and a translation structure as described above.
We will denote a translation surface by capitol letters such as $R$, $S$ or $T$,
and points on the translation surface by lower case letters such as $r$, $s$, or $t$.
A translation surface $S=\pi^{-1}([X, x_0, \sA])$
has a canonical basepoint, $o_S=[X,x_0,\sA;x_0]$. 
We will identify each point in $\sM$ with the translation surface above the point.
This allows us to write $S \in \sM$ without reference to equivalence classes.
Note that $S \subset \sE$, so points in the translation surface $S$ also belong to $\sE$.
But, we will only rarely want to refer to a point $s \in \sE$ without referring to the translation surface
$S=\pi(s)$ on which the point lies. Frequently, we will redundantly write points of $\sE$ as
pairs $(S,s) \in \sE$ where $S=\pi(s)$. 
\end{convention}

\section{Immersions and topologies}
\name{sect:immersions}

\subsection{Definition of immersion}
Let $S$ be a translation surface. We let $\PC(S)$ denote the collection of all path-connected subsets of $S$ that contain the basepoint $o_S$. 

Let $S$ and $T$ be a translation surfaces.
Let $A \in \PC(S)$ and let $B \subset T$ be an arbitrary subset. An {\em immersion} 
of $A$ into $B$ is a continuous map 
$\iota:A \to B$ which respects the basepoint and the translation structures. That is, we require:
\begin{itemize}
\item $\iota(o_S)=o_T$.
\item For all $s \in A$, there is a choice 
chart $(U,\phi)$ in the maximal translation atlas of $S$, with $s \in U$,
a chart $(V, \psi)$ in the maximal translation atlas of $T$,
and a vector $\v \in \R^2$ so that
$\iota(U \cap A) \subset V$ and
$$\psi \circ \iota(s')=\v+\phi(s') \quad
\text{for all $s' \in U \cap A$}.$$
\end{itemize}
If there is an immersion
of $A$ into $B$, we say $A$ {\em immerses} in $B$ and write $A \imm B$. We write
$\exists \iota:A \imm B$ to represent the statement ``there is an immersion, $\iota$
from $A$ into $B$.'' If there is no such immersion, we write $A \not \imm B$.

\begin{proposition}
If there is an immersion $\iota:A \imm B$, then it is unique.
\end{proposition}
\begin{proof}
Since $A$ is path-connected and our immersion must respect the basepoint, the immersion is determined by analytic continuation.
\end{proof}

An injective immersion is called an {\em embedding}. We denote the statement ``there exists an embedding $e$ of $A$ into $B$'' by
$\exists e:A \emb B$. We follow the same notational scheme as for immersions.

\begin{corollary}
Suppose $A \in \PC(S)$ and $B \in \PC(T)$. If $A \imm B$ and $B \imm A$,
then both immersions are embeddings, and the two embeddings are inverses of one another.
\end{corollary}
\begin{proof}
Both the identity map on $A$ and the composition $A \imm B \imm A$ are immersions. They
are identical, because of the uniqueness of immersions. The conclusion follows. 
\end{proof}

If the statement of the corollary is satisfied for $A$ and $B$, then we say these are {\em isomorphic subsets} of translation surfaces. Isomorphic subsets are indistinguishable from the point of view of immersions and embeddings:

\begin{corollary}
\name{prop:indistinguishable}
\com{Used.}
Let $A \in \PC(S)$ and $B \subset T$.
The truth of the statements $A \imm B$ and $A \emb B$ do not depend
on the choice of representative of $A$ from its isomorphism class. 
If $B \in \PC(T)$, the same holds for the choice of $B$ from its isomorphism class.
\end{corollary}

\subsection{\texorpdfstring{The immersive topology on $\sM$}{The immersive topology on M}}
\name{sect:top M}
Let $\PC$ denote the collection of all path connected subsets of translation surfaces that contain the basepoint. 
An {\em open disk} is a set in $\PC$ that is homeomorphic to an open disk.
A {\em closed disk} is a set in $\PC$ that is homeomorphic to a closed disk and contains the basepoint in
its interior.
We denote the collection of all closed disks in a translation surface $S$ by $\cdisk(S)$,
and the set of all open disks in $S$ by $\disk(S)$. However, 
we will frequently refer to open and closed disks without referring to the surface which contains them.

The {\em immersive topology} on $\sM$ is the coarsest topology so that the following list of sets are all open:
\begin{itemize}
\item Sets of the form $\sM_{\imm}(D)=\{S \in \sM~:~D \imm S\},$ where $D$ is a closed disk.
\item Sets of the form $\sM_{\not \imm}(U)=\{S \in \sM~:~U \not \imm S\}$, where $U$ is an open disk.
\item Sets of the form 
$$\sM_+(D,U)=\{S \in \sM~:~\text{$\exists \iota:D \imm S$ and $o_S \in \iota(U)$}\},$$
where $D$ is a closed disk and $U$ is an open subset of the interior, $D^\circ$.
\item Sets of the form
$$\sM_{-}(D,K)=\{S \in \sM~:~\text{$\exists \iota:D \imm S$ and $o_S \not \in \iota(K)$}\},$$
where $D$ is a closed disk and $K \subset D$ is closed.
\end{itemize}

We will find the following results useful:
\begin{theorem}[Embedding Theorem]
\name{thm:embedding}
If $D$is a closed disk, then $\sM_{\emb}(D)$ is open.
\end{theorem}

\begin{theorem}[Disjointness Theorem]
\name{thm:disjointness}
If $D \in \PC$ is a closed disk, and $K_1$ and $K_2$ are disjoint closed subsets of $D$, then the following set
is open in $\sM$:
$$\sM_{\emptyset}(D;K_1,K_2)=\{S \in \sM~:~ \text{$\exists \iota:D \imm S$ and $\iota(K_1) \cap \iota(K_2)=\emptyset$}\}.$$
\end{theorem}

\subsection{\texorpdfstring{The immersive topology on $\sE$}{The immersive topology on E}}
We will define a topology on $\sE$ which builds off of the topology we defined on $\sM$ above.
The {\em immersive topology} on $\sE$ is the coarsest topology so that 
the projection $\pi:\sE \to \sM$ is continuous and so that the set 
$$\sE_+(D,U)=\{(S,s) \in \sE~:~\text{$\exists \iota:D \imm S$ and $s \in \iota(U)$}\}$$
is open whenever $D$ is a closed disk
and $U$ is an open subset of its interior $D^\circ$.

Later in the paper, we will prove the following:
\begin{proposition}
\name{prop:E-}
If $D$ is a closed disk and $K \subset D$ is closed, then the following set is open:
$$\sE_{-}(D,K)=\{(S,s) \in \sE~:~\text{$\exists \iota:D \imm S$ and $s \not \in \iota(K)$}\}.$$
\end{proposition}

\section{Main Results}
\name{sect:results}
The following result guarantees that the spaces $\sM$ and $\sE$ are fairly reasonable topological spaces. In particular,
limits are unique.

\begin{theorem}
\name{thm:immersive}
The immersive topologies on $\sM$ and $\sE$ are second countable and Hausdorff.
\end{theorem}

We find it useful to observe that immersions of open disks vary continuously in both the 
domain of the immersion and the choice of the target:
\begin{proposition}[Joint continuity of immersions]
\name{prop:continuity immersions}
Let $U$ be an open disk,
and let $\sI(U) \subset \sM$ denote those $S \in \sM$ so that $U \imm S$. For $S \in \sI(U)$,
let $\iota_S:U \imm S$ be the associated immersion. Then, the following map is continuous:
$$I_U:\sI(U) \times U \to \sE; \quad (S,u) \mapsto \iota_S(u)$$ 
\end{proposition}

Since our surfaces are pointed, it is reasonable to ask what happens when the basepoint is moved. There are two important maps related to this idea. 
First, if $S \in \sM$ and $s \in S$, then we define $\BC(S,s)=S^s \in \sM$ to be the 
translation surface which is isomorphic to $S$ with the basepoint relocated to $s$. This defines
the {\em basepoint changing map} $\BC:\sE \to \sM$. 
\begin{theorem}
\name{thm:basepoint changing map}
The basepoint changing map $\BC:(S,s) \mapsto S^s$ is continuous.
\end{theorem}
In addition to the basepoint changing map, there is a {\em basepoint changing isomorphism}
$\beta_s: S \to S^s$. This is the translation isomorphism which sends $s \in S$ to the basepoint of $S^s$.
\begin{theorem}
\name{thm:basepoint changing isomorphism}
Consider the basepoint changing isomorphism $s' \mapsto \beta_s(s')$ and
the inverse basepoint changing isomorphism, which sends $t \in S^s$ to $\beta_s^{-1}(t)$.
Both maps are jointly continuous in both $s$ and the given domain ($S$ and $S^s$, respectively). 
\end{theorem}

There is a standard action of elements $A \in \GL(2,\R)$ on a translation structures on a surface $X$. If
$\sA=\{(U_j, \phi_j): j \in \sJ\}$ is an atlas determining a translation structure on $X$, then the image of this structure under $A$ is given by
$$A(\sA)=\{(U_j, A \circ \phi_j) : j \in \sJ\}.$$ The fact that $A(\sA)$ is also
a translation structure on $X$ follows from the observation that 
the group of translations of the plane is a normal subgroup of the affine group of the plane,
which also contains $\GL(2,\R)$. The $\GL(2,\R)$ action sends isomorphic translation structures to isomorphic translation
structures, and thus induces an action on $\sM$. Furthermore, there is a natural action of each $A\in \GL(2,\R)$ on $\sE$ given by 
$$A: \sE \to \sE; \quad [X,x_0,\sA;x] \mapsto [X,x_0,A(\sA);x].$$
Here, we are choosing a representative 
$(X,x_0,\sA;x)$ of $[X,x_0,\sA;x] \in \sE$, but the image is independent of this choice.
Observe that this action restricted to a fiber of $\pi:\sE \to \sM$ restricts
to a homeomorphism between each translation surface and the image surface.
That is, if $S \in \sM$ is a translation surface, then 
$A|_S:S \to A(S)$ is a homeomorphism. 
The action is natural in the sense that for any chart $(U,\phi)$ for the translation structure
on $S$, there is a chart $(V,\psi)$ for the translation structure on $A(S)$ so that
$A(U)=V$ and the following diagram commutes:
\begin{equation}
\name{eq:A}
\begin{tikzcd}[column sep=scriptsize, row sep=scriptsize]
U \arrow{r}{A|_{S}} \arrow{d}{\phi} & V \arrow{d}{\psi} \\
\R^2 \arrow{r}{A} & \R^2
\end{tikzcd}
\end{equation}

\begin{theorem}[Continuity of affine actions]
\name{thm:affine}
The actions of $\GL(2,\R)$ on $\sM$ and $\sE$ are continuous.
\end{theorem}

An {\em affine automorphism} of a translation surface $S$ is a homeomorphism $S \to S$ which respects the affine
structure underlying the translation structure on $S$. More concretely, a homeomorphism 
$h:S \to S$ is an affine automorphism if there is an $A \in \GL(2,\R)$ so that 
for each chart $(U,\phi)$ in the maximal translation atlas 
on $S$, the pair $\big(h(U),A \circ \phi\big)$ is also a chart in this atlas. 
We call $A \in \GL(2,\R)$ the {\em derivative} of the affine automorphism $h$. 

Note that we do not require an affine automorphism $h:S \to S$ to respect the basepoint.
If $s=h^{-1}(o_s)$ is the preimage of the basepoint, then $A(S^s)=S$ and $h$ is given by the composition
$$S \xrightarrow{\beta_s} S^s \xrightarrow{A} S,$$
where $\beta_s$ is the basepoint changing isomorphism. By continuity of the $\GL(2,\R)$ actions and joint continuity of $\beta_s$ and $\beta_s^{-1}$, we have obtained the following:

\begin{corollary}[Convergence of affine automorphisms]
Let $\langle S_n \in \sM \rangle$ be a sequence of translation surfaces converging to a surface $S \in \sM$.
Suppose that each $S_n$ admits an affine homeomorphism $h_n$ with derivative $A_n$. Further suppose that
$\langle A_n \rangle$ converges to some $A \in \GL(2,\R)$, and that $s_n=h_n^{-1}(o_{S_n})$ converges to some limit point $s \in S$. 
Then, there is an affine homeomorphism $h:S \to S$ with derivative $A$ so that $h(s)=o_S$. Moreover, for any sequence $\langle t_n \in S_n \rangle$
tending to $t \in S$, the sequence $\langle h_n^k(t_n) \rangle$ converges to $h^k(t)$ for all $k \in \Z$.
\end{corollary}

Finally, we will prove that the collection of surfaces containing an open disk
is compact: 

\begin{theorem}[Compactness]
\name{thm:compactness}
Let $U$ be an open disk in a translation surface containing the basepoint. The following set is compact:
$$\sM \smallsetminus \sM_{\not \emb}(U)=\{S \in \sM~:~U \emb S\}.$$
\end{theorem}
\section{General Approach}
\name{sect:approach}
The approach of this paper is to use results already proved in \cite{HooperImmersions1}
about the immersive topology on translation structures on disks. In this section, we explain some
results which makes this approach work.

We begin by describing the philosophy of the approach.
Each translation surface $S \in \sM$, has a universal cover $\tilde S$ which inherits a translation structure
by pulling back the structure along the covering map $p_S:\tilde S \to S$. Let $(\Delta,x_0)$ be the pointed
disk and let $\tM=\sM(\Delta,x_0) \subset \sM$ be the collection
of all translation structures on the disk. Thus, $\tilde S$ lies in $\tM$, and we can recover $S$ as the quotient of $\tilde S$ modulo the deck group of the cover. So, another way to think of a point in $\sM$ is as a choice of a point in $\tM$
together with a discrete group of translation automorphisms of $\tM$. This approach brings to mind the point of view of 
the universal family of curves over moduli space. See \cite{Zvonkine12} for background. 

A third viewpoint on the moduli space $\sM$ is that we can associate a point $S \in \sM$ to its universal
cover $\tilde S \in \tM$ and the preimages of the basepoint under the covering map, $p_S^{-1}(o_S)$. The set
$p_S^{-1}(o_S)$ is a discrete subgroup of $\tilde S$ with the property that $\BC(\tilde S,\tilde s)=\tilde S$
for each $\tilde s \in p_S^{-1}(o_S)$. We think of $p_S^{-1}(o_S)$ as a subset of $\tE \subset \sE$, the canonical disk bundle over $\tM$. From this point of view, the topology on $\sM$ can be though of the ``geometric limit topology'' induced by the topologies on $\tM$ and $\tE$. That is, a sequence of translation surfaces $S_n$ with basepoints $o_n$
should converge to $S$ if the universal covers converge and the sequence of sets $p_n^{-1}(o_n) \subset \tilde S_n$
converge to $p_S^{-1}(o_S)$ within $\tE$. This point of view is formalized by Theorem \ref{thm:convergence in M}
below.

In the remainder of the section, we formally state results related to the ideas introduced above.

\begin{theorem}
\name{thm:universal cover}
The map $\sM \to \tM$ which sends a translation surface $S$ to its universal cover $\tilde S$
is continuous.
\end{theorem}

For a translation surface $S$, let $p_S:\tilde S \to S$ be the covering map. The domain of this map can most broadly be considered to be
\begin{equation}
\name{eq:projection domain}
\sP=\{(S,\tilde s) \in \sM \times \tE~:~\tilde s \in \tilde S\}.
\end{equation}
We say the {\em covering projection}
is the map
$$p: \sP \to \sE; \quad (S,\tilde s) \mapsto p_S(\tilde s).$$

\begin{theorem}[Projection Theorem]
\name{thm:covering projection}
The covering projection is continuous.
\end{theorem}

Finally, we state results describing convergence criteria and consequences of convergence in both $\sM$ and $\sE$.

\begin{theorem}[Convergence in $\sM$]
\name{thm:convergence in M}
Let $\langle S_n \in \sM\rangle$ be a sequence of translation surfaces with basepoints $o_n \in S_n$. 
Let $p_n: \tilde S_n \to S_n$ be the universal covering maps.
Then, the sequence $\langle S_n \rangle$ converges if and only if the sequence of universal covers $\tilde S_n \in \tM$ converge to some $\tilde S \in \tM$ and there is a discrete set of points $\tilde O \subset \tilde S$ such that the following statements hold:
\begin{enumerate}
\item For every $\tilde o \in \tilde O$, there is sequence $\langle \tilde o_n \in p_n^{-1}(o_n)\rangle$
converging to $\tilde o$ in $\tE$. 
\item For every increasing sequence of integers $\langle n_k \rangle$ and every 
sequence of points $\langle \tilde o_{n_k} \in p_{n_k}^{-1}(o_{n_k})\rangle$ which converges to some point 
$\tilde o \in \tE$,
we have $\tilde o \in \tilde O$.
\end{enumerate}
Moreover, if these statements hold, then $\tilde S$ is the universal cover of $S=\lim_{n \to \infty} S_n$,
and $\tilde O=p_S^{-1}(\tilde o_S)$.
\end{theorem}

\begin{theorem}[Convergence in $\sE$]
\name{thm:convergence in E}
Let $\langle (S_n,s_n) \in \sE\rangle$ be a sequence, and let $(S,s) \in \sE$. 
Then, $\langle (S_n,s_n) \rangle$ converges
to $(S,s)$ if the following statements are satisfied:
\begin{enumerate}
\item The sequence $\langle S_n \in \sM \rangle$ converges to $S \in \sM$. 
\item There is a sequence $\langle \tilde s_n \in p_n^{-1}(s_n) \subset \tE \rangle$ which converges
to a point in $p^{-1}(s) \subset \tE$. 
\end{enumerate}
Conversely, if $\langle (S_n,s_n) \rangle$ converges
to $(S,s)$, then $S_n$ converges to $S$ and for each $\tilde s \in p_S^{-1}(s)$,
there is a sequence  $\langle \tilde s_n \in p_n^{-1}(s_n) \rangle$ which converges to $\tilde s$. 
\end{theorem}

\section{Outline of remainder of paper}
\name{sect:outline}
In \S \ref{sect:cover}, we investigate how a translation structure on a surface gives rise to a translation structure on its universal cover, and explore how this idea interacts with immersions.

In \S \ref{sect:countability}, we prove that the immersive topologies on $\sM$ and $\sE$ are second countable. This proves part of Theorem \ref{thm:immersive}.

In \S \ref{sect:embeddings}, we investigate embeddings. We prove Theorems \ref{thm:embedding} and \ref{thm:disjointness}
as well as Proposition \ref{prop:E-}.

In \S \ref{sect:disk}, we further investigate the connection between a surface and its universal cover. We prove all
the results in \S \ref{sect:approach},
and prove results which will allow us to use work done in \cite{HooperImmersions1} for the remainder of the article.
We also prove that the immersive topologies on $\sM$ and $\sE$ are Hausdorff, completing the proof of Theorem \ref{thm:immersive}.

In \S \ref{sect:basepoint}, we investigate the action of moving the basepoint. We prove Theorems \ref{thm:basepoint changing map} and \ref{thm:basepoint changing isomorphism}.

In \S \ref{sect:affine}, we prove that the $\GL(2,\R)$ actions on $\sM$ and $\sE$ are continuous.

In \S \ref{sect:compactness}, we prove Theorem \ref{thm:compactness}, which states that the set of surfaces
into which an open disk embeds is compact.
\section{The universal cover and immersions}
\name{sect:cover}

\subsection{Developing map}
\name{sect:developing map}
Let $\Delta$ denote the open topological disk with basepoint $\tilde x_0$. 
Suppose that $\tilde \sA$ is a translation atlas on $\Delta$. Then by analytic continuation, there is a unique
local homeomorphism $\mathit{dev}:\Delta \to \R^2$ so that:
\begin{itemize}
\item $\mathit{dev}(\tilde x_0)=\0$. 
\item For each chart $(\tilde U,\tilde \phi) \in \tilde \sA$, the map
$\tilde \phi$ differs from $\mathit{dev}|_{\tilde U}$ locally only by translation. That is,
for each $\tilde x \in \tilde U$, 
there is an open neighborhood $\tilde V$ of $x$ inside $\tilde U$
so that $\tilde \phi_{\tilde V}$ and $\mathit{dev}|_{\tilde V}$ differ by translation.
\end{itemize}
We call the map $\mathit{dev}$ the {\em developing map} of the translation structure.

Let $\tM=\sM(\Delta,\tilde x_0)$ and $\tE=\sE(\Delta,\tilde x_0)$. 
Following \cite{HooperImmersions1}, we call the fibers of the restricted projection
$$\tpi=\pi|_{\tE}:\tE \to \tM$$
{\em planar surfaces}, and denote them by letters such as $P$ and $Q$. An alternate definition
is that a planar surface is a translation surface which is homeomorphic to a disk.
Each planar surface has an associated developing map. The union of these maps gives the {\em bundle-wide developing map} $\dev:\tE \to \R^2$. We denote the individual developing maps of planar surfaces by restriction: $\dev|_P:P \to \R^2$.

\subsection{The universal cover}
A translation structure $(X,x_0,\sA)$ induces a translation structure on the universal cover
$(\tilde X,\tilde x_0)$ of $(X,x)$. Let $p:\tilde X \to X$ be the covering projection which satisfies
$p(\tilde x_0)=x_0$. Given the translation atlas $\sA$ on $(X,x_0)$, consider the atlas
\begin{equation}
\name{eq:lifted charts}
\big\{(\tilde U,\phi \circ p)~:~\text{$\tilde U=\pi^{-1}(U)$ and $(U,\phi)\in \sA$}\big\}.
\end{equation}
It can be observed that this new atlas is also a translation atlas, and thus it can be extended to a unique
maximal translation atlas, which we denote by $\tilde \sA$.

If $S \in \sM$ is a translation surface (with basepoint $o_S$), its universal cover also inherits a translation structure as above. This universal cover is therefore isomorphic to a unique planar surface which we denote by $\tilde S \in \tM$. We denote the basepoint of $\tilde S$ by $\tilde o_S$. 
The isomorphism can be used to produce a covering map $p_S:\tilde S \to S$ which satisfies
$p_S(\tilde o_S)=o_S$. 

\begin{remark}
We can recover a translation structure on a surface $(X,x_0)$ from the developing 
map $\mathit{dev}:\tilde X \to \R^2$. Indeed, an open set
$\tilde U \subset \tilde X$ so that $p_S|_{\tilde U}:\tilde U \to X$ is a homeomorphism onto its image,
the pair $\big(p_S(\tilde U), \mathit{dev} \circ (p_S|_{\tilde U})^{-1}\big)$ is a compatible chart.
The maximal collection of such charts recovers the translation structure.
We view the developing map as easier to work than the translation atlas, but any statement
which can be made in terms of the developing map can also be made in terms of the translation atlas.
\end{remark}

In this paper, we will not be interested in immersing all types of subsets of translation surfaces. We will primarily be interested in 
$$\SC(S)=\{A \in \PC(S)~:~\text{$A$ is locally path-connected and simply connected}\}.$$
We explain in some propositions why immersions are more natural with respect to sets in $\SC(S)$. 

\begin{proposition}
\name{prop:lift}
If $A \in \SC(S)$, then
there is an embedding $\ell_A:A \emb \tilde S$. Moreover, the composition
$p_S \circ \ell_A$ is the identity on $A$. 
\end{proposition}
We call the map $\ell_A$ the {\em lifting map} and call $\tilde A=\ell_A(A)$
the {\em lift} of $A$.

\begin{proof}
Since $A$ is connected and locally path-connected, it admits a universal cover $\tilde A$. By general covering space theory, the inclusion of $A$ into $S$ lifts to a map $\tilde A \to \tilde S$ so that $\tilde o_S \mapsto o_S$. Since $A$ is simply connected, we can identify $\tilde A$ with $A$. This gives our map $\ell_A:A \to \tilde S$. It is an immersion because $p_S \circ \ell_A$ is the inclusion of $A$ into $S$. In particular, it respects the translation structure defined by the charts given in equation \ref{eq:lifted charts}.
\end{proof}

\begin{proposition}
\name{prop:lifted immersion}
Let $S$ and $T$ be a translation surfaces.
Let $A \in \SC(S)$ and let $B \subset T$ be an arbitrary subset containing the basepoint. The following statements are equivalent:
\begin{enumerate}
\item There is an immersion $\iota:A \imm B$.
\item There is an immersion $\tilde \iota:\tilde A \to \tilde T$ with $p_T \circ \tilde \iota(\tilde A) \subset B$.
\item There is a continuous map $\tilde \iota:\tilde A \to \tilde T$ with $p_T \circ \tilde \iota(\tilde A) \subset B$ so that
$\tilde \iota(\tilde o_S)=\tilde o_T$ and
 $\dev|_{\tilde S}(\tilde s)=\dev|_{\tilde T} \circ \tilde \iota(\tilde s)$ for all $\tilde s \in \tilde A$.
\end{enumerate}
Moreover, the maps $\tilde \iota$ in statements (2) and (3) are the same.
\end{proposition}
The situation resulting from the existence of an immersion $\iota:A \imm B$ with $A \in \SC(S)$ is summarized by the following commutative diagram:
\begin{center}
\begin{tikzcd}[row sep=tiny, column sep=normal]
& \tilde A \arrow{dl}[above,sloped]{\dev|_{\tilde S}} \arrow[squiggly]{dd}{\tilde \iota} \arrow[yshift=0.2em]{r}{p_S} & A \arrow[squiggly]{dd}{\iota} \arrow[yshift=-0.2em]{l}[below]{\ell_A} \\
\R^2  & \\
& \tilde T \arrow{ul}[below,sloped]{\dev|_{\tilde T}} \arrow{r}{p_T} & T
\end{tikzcd}
\end{center}
We call the map $\tilde \iota$ the {\em lifted immersion}.

\begin{proof}
Statements (2) and (3) are equivalent, because the translation structures are completely determined by the single charts
$(\tilde S,\dev|_{\tilde S})$ and $(\tilde T, \dev|_{\tilde T})$. Statement (3) simply restates the definition of immersion but restricted to consider these charts. These statements imply (1) since $p_T \circ \tilde \iota \circ \ell_A$ is an immersion of $A$ into $B$. Finally, given $\iota$, we can construct $\tilde \iota$ by following the proof of the prior proposition with $\iota$ replacing the inclusion of $A$ into $S$.
\end{proof}
\section{Second countability}
\name{sect:countability}
Our proof of second countability of the topologies on $\sM$ and $\sE$ essentially follows from work
done in \S \ref*{I-sect:rectangular unions} of \cite{HooperImmersions1}. We offer an explicit countable subbasis for the immersive topologies on $\sM$ and $\sE$.

We recall some definitions from \S \ref*{I-sect:rectangular unions} of \cite{HooperImmersions1}.
An open (resp. closed) {\em rectangle} $R$ in a planar surface $P$ is a subset so that the developing map
restricted to $R$ is a homeomorphism onto an open (resp. closed) rectangle in $\R^2$.
We call a rectangle {\em rational} if the vertices of the image rectangle lie in $\Q^2$. 
An open (resp. closed) {\em rational rectangular union} in $P$ is a finite union of open (resp. closed) rational rectangles in 
$P$ which is connected and whose boundary consists of a union of closed curves.

\begin{theorem}[Second countability of $\sM$]
The collection of subsets of $\sM$ of the following four types give a countable subbasis for the immersive topology:
\begin{itemize}
\item Sets of the form $\sM_{\imm}(D)$, where $D \in \cdisk$ is a rational rectangular union.
\item Sets of the form $\sM_{\not \imm}(U)$, where $U \in \disk$ is a rational rectangular union.
\item Sets of the form $\sM_{+}(D,U)$, where $D \in \cdisk$ is a rational rectangular union
and $U \subset D^\circ$ is an open rational rectangle.
\item Sets of the form $\sM_-(D,K)$, where $D \in \cdisk$ is a rational rectangular union
and $K \subset D$ is a finite union of closed rational rectangles.
\end{itemize}
\end{theorem}
\begin{proof}
The collection of all isomorphism classes of rational rectangular unions in planar surfaces is countable by 
Corollary \ref*{I-cor:countably many classes} of \cite{HooperImmersions1}. It follows that the subbasis
described in the theorem is countable. It remains to prove that the subbasis above generates the immersive
topology on $\sM$. For the proof, we will call the topology generated by the sets listed in the theorem
the {\em subbasis topology}. Clearly every open set in the subbasis topology is open in the immersive topology.
We will prove that every open set used to define the immersive topology on $\sM$
is open in the subbasis topology. 

First consider $\sM_{\imm}(D_1)$, where $D_1$ is an arbitrary closed disk. Suppose that
$S \in \sM_{\imm}(D_1)$ so that there is an immersion $\iota:D_1 \imm S$. By Proposition \ref{prop:lifted immersion},
this is equivalent to the existence of an immersion $\tilde \iota:D_1 \imm \tilde S$. We can then apply Theorem \ref*{I-thm:nested} of \cite{HooperImmersions1} to produce a rectangular union $D_2 \in \cdisk(\tilde S)$ so that
$\tilde \iota(D_1) \subset D_2^\circ$. By transitivity of immersions, $D_1$ immerses in a surface whenever $D_2$ does.
Also, $S \in \sM_{\imm}(D_2)$ since the covering map $p_S:\tilde S \to S$ restricts to an immersion $p_S|_{D_2}:D_2 \imm S$. Thus, $S \in \sM_{\imm}(D_2) \subset \sM_{\imm}(D_1)$, proving that the later is open in the subbasis topology.

Now suppose $S \in \sM_+(D_1,U_1)$, where $D_1$ is an arbitrary closed disk and $U_1 \subset D_1$ is an arbitrary open subset. We repeat the argument above to construct $D_2$. Observe that $\iota=p_S \circ \tilde \iota$. 
Thus, there is a point $\tilde o \in \tilde \iota(U_1)$ so that $p_S(\tilde o)$ is the basepoint $o_S$ of $S$. 
Since $\iota(U_1)$ is open, we can find an open rational rectangle $U_2$ satisfying $\tilde o \in U_2 \subset \tilde \iota(U_1)$.
Thus, $S \in \sM_{+}(D_2,U_2) \subset \sM_{+}(D_1,U_1)$, proving the later is open.

Now suppose that $S \in \sM_{-}(D_1,K_1)$, where $K_1 \subset D_1$ is an arbitrary closed set. 
Construct $D_2$ as above, and note that $o_S \not \in p_S \circ \tilde \iota(K_1)$. 
Equivalently, we have $\tilde \iota(K_1) \cap p_S^{-1}(o_S)=\emptyset$. 
The basepoint $o_S$ has an $\epsilon$ neighborhood which isometric to a Euclidean disk. Then, no image of a metric ball of radius less than $\frac{\epsilon}{2}$ in $\tilde S$
under $p_S$ can contain more than one point in $p_S^{-1}(o_S)$. Using compactness, we conclude that 
$p_S^{-1}(o_S) \cap D_2$ is finite. 
Since $\tilde \iota(K_1) \subset D_2^\circ$,
given any point $\tilde s \in \tilde \iota(K_1)$, we can find a closed rational rectangle $R$ which
does not intersect $p_S^{-1}(o_S)$ and whose interior contains $\tilde s$. The interiors of these rectangles
describe an open cover of the compact set $\tilde \iota(K_1)$, so there is a finite collection $\sR$ of such rectangles
which cover $\tilde \iota(K_1)$. Let $K_2$ be the union of these rectangles, which contains $\tilde \iota(K_1)$
and do not intersect $p_S^{-1}(o_S)$. Then, 
$S \in \sM_{-}(D_2,K_2) \subset \sM_{-}(D_1,K_1)$, which proves that $\sM_{-}(D_1,K_1)$ is open.

Finally, suppose that $S \in \sM_{\not \imm}(U_1)$, where $U_1$ is an open disk. Then $U_1 \not \imm S$,
or equivalently $U_1 \not \imm \tilde S$. 
By using a closed disk family for $U_1$, we can find a $U_2 \subset U_1$
with compact closure so that $U_2 \not \imm S$. (See Propositions \ref*{I-prop:embedding} and \ref*{I-prop:closed disk family}
of \cite{HooperImmersions1}.) Then, we can find an open rational rectangular union $U_3 \subset U_1$ 
which is homeomorphic to a disk so that
$U_2 \subset U_3$ using Theorem \ref*{I-thm:nested} of \cite{HooperImmersions1}. We have $U_3 \not \imm S$,
since such an immersion would restrict to an immersion of $U_2$ into $S$. Similarly,
whenever $U_1$ immerses into a surface, so does $U_3$. Thus,
$S \in \sM_{\not \imm}(U_3) \subset \sM_{\not \imm}(U_1)$. This proves that $\sM_{\not \imm}(U_1)$ is open.
\end{proof}

\begin{theorem}[Second countability of $\sE$]
A countable subbase for the immersive topology on $\sE$ is given by preimages under $\pi:\sE \to \sM$
of a countable subbase for the topology on $\sM$ together with sets of the form 
$\sE_+(D,U)$, where $D \in \cdisk$ is a rational rectangular union and $U \subset D^\circ$ is an
open rational rectangle.
\end{theorem}
We omit the proof, because it very similar to the prior result. In particular,
the proof that $\sE_+(D_1,U_1)$ is open in the topology generated by this subbasis is identical
to the argument that $\sM_+(D_1,U_1)$ was open in the prior proof.

\section{Embeddings}
\name{sect:embeddings}
In this section, we prove the Embedding Theorem (Theorem \ref{thm:embedding}), which stated that sets of the form
$\sM_{\emb}(D)$ are open whenever $D$ is a closed disk in a translation surface. We also show that Proposition \ref{prop:E-} (that $\sE_-(D,K)$ is open in $\sE$) follows.

\subsection{Ball embeddings}
Let $B_\epsilon \subset \R^2$, be the open ball of radius $\epsilon$ centered at the origin, with the origin considered as the basepoint. Given a point $s$ in a translation surface $S$, we define the {\em embedding radius}
to be
$$\ER(S,s)=\max \{\epsilon>0~:~B_\epsilon \emb S^s\}.$$
Here, $S^s$ denotes the translation surface which is translation isomorphic to $S$, but has the basepoint in the location of $s \in S$. See \S \ref{sect:results}. This maximum is well defined unless $S$ is the plane, in which case we take $\ER(S,s)=\infty$. 

Suppose $\epsilon<\ER(S,s)$. Then, there is an embedding $e:B_\epsilon \emb S^s$. Consider the basepoint changing isomorphism
$\beta_s:S \to S^s$, which is a translation isomorphism and sends $s \in S$ to the basepoint of $S^s$. We define the {\em ball embedding} 
$$\BE_s:B_\epsilon \to S; \quad \v \mapsto \beta_s^{-1} \circ e(\v).$$

As in \cite[\S \ref*{I-sect:Embedding radii}]{HooperImmersions1},
we also define some related quantities. If $U \subset S$ is open, then it is translation isomorphic to a translation surface $\hat U$, and there is a natural embedding $e:\hat U \to U$. For $s \in U$, we define
$$\ER(s \in U)=\ER\big(\hat U,e^{-1}(s)\big).$$
If $K \subset U$ is compact, then we define
$$\ER(K \subset U)=\min_{s \in K} \ER(s \in U).$$
This minimum is well defined, because the function $s \mapsto \ER(s \in U)$ is $1$-Lipschitz.

\subsection{Proof of the Embedding Theorem}
We now prove the Embedding Theorem. We will briefly describe the main idea of the proof.
Pick a closed disk $D$.
The basic plan of the proof is to show that the Embedding Theorem is satisfied
locally. By this we mean for every pair of points $u$ and $v$ in $D$, there are respective neighborhoods
$U$ and $V$ and an open subset $\sU \subset \sM$
so that if $S \in \sI$, then there is an immersion $\iota: D \imm S$ and this immersion looks like an embedding
if we restrict attention to both $U$ and $V$. This statements is made rigorous by the two lemmas below.
The proof concludes by making an appeal to compactness of $D \times D$.

\begin{proposition}
\name{prop:13}
Let $D$ be a closed disk in a planar surface $P$.
Let $\epsilon>0$ be such that there is an embedding $e:\bar B_\epsilon \imm D$.
Then, the following set is open in $\sM$:
$$\big\{S \in \sM~:~\text{$\exists \iota:D \imm S$ and $(\iota \circ e)^{-1}(o_S) \cap \bar B_\epsilon=\{o_P\}$}\big\}.$$
\end{proposition}
\begin{proof}
Choose $\epsilon \leq \ER(o_P \in D^\circ)$.
Then, there is an embedding $e:\bar B_\epsilon \emb D$. Let $\sX \subset \sM$ be the set defined
at the end of the proposition. Define 
$$\sY=\sM_{-}(D, \bar B_\epsilon \smallsetminus B_{\frac{\epsilon}{2}}).$$
We claim that $\sX = \sY$. This will prove the theorem, since $\sY$ is a subbasis element of the topology on $\sM$.

Observe that $\sX \subset \sY$, since $S \in \sX$ implies that there is an immersion $\iota:D \imm S$ and 
whenever $p \in P$ is in $e(\bar B_\epsilon)\smallsetminus \{o_P\}$, $\iota(p) \neq o_S$. This in particular
holds for points $p \in \bar B_\epsilon \smallsetminus B_{\frac{\epsilon}{2}}.$

Now suppose that $S \in \sY$ and $S \not \in \sX$. Then there is an immersion $\iota:D \imm S$,
and there is a $\v \in \bar B_\epsilon \smallsetminus \{\0\}$ so that $\iota \circ e(\v)=o_S$. 
Since $S \in \sY$, it must be that $\v \in B_{\frac{\epsilon}{2}} \smallsetminus \{\0\}$.
The collection $\{\iota \circ e(t \v)~:~0 \leq t \leq 1\}$ is a closed geodesic on $S$.
It follows that so long as $n \in \Z$ and $|n \v| \leq \epsilon$, we have
$\iota \circ e(n \v)=o_S$. Choose $n$ to be the integer satisfying 
$$\frac{\epsilon}{2} \leq n|\v|<\frac{\epsilon}{2}+|\v|<\epsilon.$$
Then, $n\v \in \bar B_\epsilon \smallsetminus B_{\frac{\epsilon}{2}}$ and
$\iota \circ e(n \v)=o_S$. By definition $S \not \in \sY$, which is a contradiction.
\end{proof}

\begin{lemma}
\name{lem:2}
Let $P$ be a planar surface, let $D_1, D_3 \in \cdisk(P)$ with $D_1 \subset D_3^\circ$.
Let $u \in D_1$. Then,
there is a closed neighborhood $\bar U \subset D_3^\circ$ of $u$ so that the following set is open:
$$\{S \in \sM~:~ \text{$\exists \iota:D_3 \imm S$ and $\iota|_{\bar U}$ is injective}\}.$$
\end{lemma}
\begin{proof}
Define the constant
$$\epsilon=\frac{1}{2}  \min \big\{\ER(D_1 \subset D_3^\circ), \ER(o_P \in D_1) \big\}.$$ 
Then, there is an embedding $e:\bar B_{2 \epsilon} \emb D_3$ be the associated embedding.
Let $\bar U \subset D_3$ be the closed $\epsilon$-ball centered at $u \in D_1$. Let $\sX \subset \sM$ be the set
defined at the end of the lemma. We must show that $\sX$ is open. In fact, we claim that $\sX=\sY$, where
$$\sY=\big\{S \in \sM~:~\text{$\exists \iota:D_3 \imm S$ and $(\iota \circ e)^{-1}(o_S) \cap \bar B_{2 \epsilon}=\{o_P\}$}\big\}.$$
This set is open by Proposition \ref{prop:13}.

We will prove that $\sX \subset \sY$. 
Suppose that $S \in \sM$ and there is an immersion $\iota:D_3 \imm S$, but that $S \not \in \sY$. We will show that $S \not \in \sX$.
Since $S \not \in \sY$, there is an $\v \in B_{2 \epsilon}$
with $\v \neq \0$ so that $\iota \circ e(\v)=o_S$. Consider the map
\begin{equation}
\name{eq:f}
f:D_1 \to D_3; \quad p \mapsto \BE_p(\v).
\end{equation}
Observe that the map $\iota \circ f:D_1 \imm S$ is an immersion since
$f(o_P)=e(\v)$. So by uniqueness of immersions $\iota \circ f=\iota$ on $D_1$. We claim that $\iota|_{\bar U}$ is not injective. To see this observe that
$$
\iota \big(\BE_u(\frac{-\v}{2})\big)=
\iota \circ f
\big(\BE_u(\frac{-\v}{2})\big)=
\iota\big(\BE_u(\frac{\v}{2})\big).$$
So, the observation that both $\BE_u(\frac{-\v}{2})$ and $\BE_u(\frac{-\v}{2})$ are both within $\bar U$ 
proves that $\iota|_{\bar U}$ is not injective.

We now will show that $\sY \subset \sX$ via a similar argument.
Let $S$ be a translation surface so that there is an immersion
$\iota:D_3 \imm S$. 
Suppose there are $p,q \in \bar U$ so that $\iota(p)=\iota(q)$.
Then, $S \not \in \sX$ and we will show $S \not \in \sY$.
Let $s \in S$ be this common image, $s=\iota(p)$.
Consider the vector $\v=e^{-1}(q) - e^{-1}(p)$, which satisfies $\v \neq 0$ and 
$|\v|<2 \epsilon$. Define $f$ as in equation \ref{eq:f}. 
We claim that $\iota \circ f$ is an immersion of $D_1$ into 
$S$. Since it is a local translation, it suffices to show that $\iota \circ f(o_P)=o_S$. 
Consider the basepoint changing isomorphism $\beta_p:P \to P^p$ and
$\beta_s:S \to S^s$. 
Let $D_1^p=\beta_p(D_1)$ and $D_3^p=\beta_p(D_3)$. 
Then, the composition
$\beta_s \circ \iota \circ \beta_p^{-1}$ is an immersion of $D_3^p$
into $S^s$. Consider precomposing with $\beta_p \circ f \circ \beta_p^{-1}:D_1^p \to D_3^p$. The composition is given by 
$\beta_s \circ \iota \circ f \circ \beta_p^{-1}$ and
is an immersion of $D_1^p$ into $S^s$. By uniqueness of immersions,
these maps agree on $D_1^p$. Thus, we have
$\iota \circ f=\iota$ on $D_1$. We conclude that
$$o_S=\iota(o_P)=\iota \circ f(o_P).$$
Since $f(o_P) \in e(B_{2 \epsilon})$, we have that $S \not \in \sX$.
\end{proof}

\begin{lemma}
\name{lem:1}
Let $P$ be a planar surface, let $D_1, D_3 \in \cdisk(P)$ with $D_1 \subset D_3^\circ$.
Let $u$ and $v$ be distinct points in $D_1$. Let $\bar U_\epsilon$ and $\bar V_\epsilon$
be the closed $\epsilon$ balls about $u$ and $v$, respectively. 
Then, for sufficiently small
$\epsilon$, the following set is open in $\sM$:
$$\{S \in \sM~:~ \text{$\exists \iota:D_3 \imm S$ and $\iota(\bar U) \cap \iota(\bar V) = \emptyset$}\}.$$
\end{lemma}

In order to prove the theorem, we will make use of the 
Fusion Theorem \cite[Theorem \ref*{I-thm:main fusion}]{HooperImmersions1}. We state a variant of this result combining 
this theorem with Proposition \ref*{I-prop:fusion disks} of \cite{HooperImmersions1}.
\begin{theorem}[Fusion Theorem]
Let $P$ and $Q$ be planar surfaces. Then, there is a unique planar surface
$R=P \fuse Q$ which satisfies the following statements:
\begin{itemize}
\item $P \imm R$ and $Q \imm R$.
\item For all trivial surfaces $S$, if $P \imm S$ and $Q \imm S$, then $R \imm S$.
\end{itemize}
\end{theorem}

\begin{proof}[Proof of Lemma \ref{lem:1}]
Choose a $D_2 \in \cdisk(P)$ so that $D_1 \subset D_2^\circ$ and $D_2 \subset D_3^\circ$. Choose
$$\epsilon<\frac{1}{2} \min \big\{\ER(o_P \in D_1^\circ), \ER(D_1 \subset D_2^\circ), \ER(D_2 \subset D_3^\circ), d(u,v)\big\}.$$
Let $\bar U \subset D_2^\circ$ be the closed ball of radius $\epsilon$ about $u$, and let $\bar V \subset D_2^\circ$ be the closed ball of radius $\epsilon$ about $v$. Because $2\epsilon< d(u,v)$, these balls are disjoint.

We will now construct a new planar surface, $Q$. Let $\beta_u:P \to P^u$ and $\beta_v:P \to P^v$
be basepoint changing isomorphisms. 
Consider the open disks
with alternate basepoints $\beta_u(D_3^\circ)$ and $\beta_v(D_2^\circ)$.  
These surfaces are isomorphic to planar surfaces, and we define $Q$ to be their fusion,
$$Q=\beta_u(D_3^\circ) \fuse \beta_v(D_2^\circ).$$
Associated to the fusion, we have immersions
$$\iota_u:\beta_u(D_2^\circ) \imm Q \and \iota_v:\beta_v(D_3^\circ) \imm Q.$$
Figure \ref{fig:embedding} depicts an example of $Q$, as well as the following objects. 
We define $q=\iota_v \circ \beta_v(o_P)$. We define
$\bar W$ to be the closed ball of radius $2 \epsilon$ about $\iota_u \circ \beta_u(o_P)$. 
We also choose a closed disk $K \in \cdisk(Q)$ which contains both $\iota_u \circ \beta_u(D_1)$
and $\iota_v \circ \beta_v(D_2)$. Note that $\bar W \subset \iota_u \circ \beta_u(D_1) \subset K$. 

\begin{figure}
\includegraphics[width=6in]{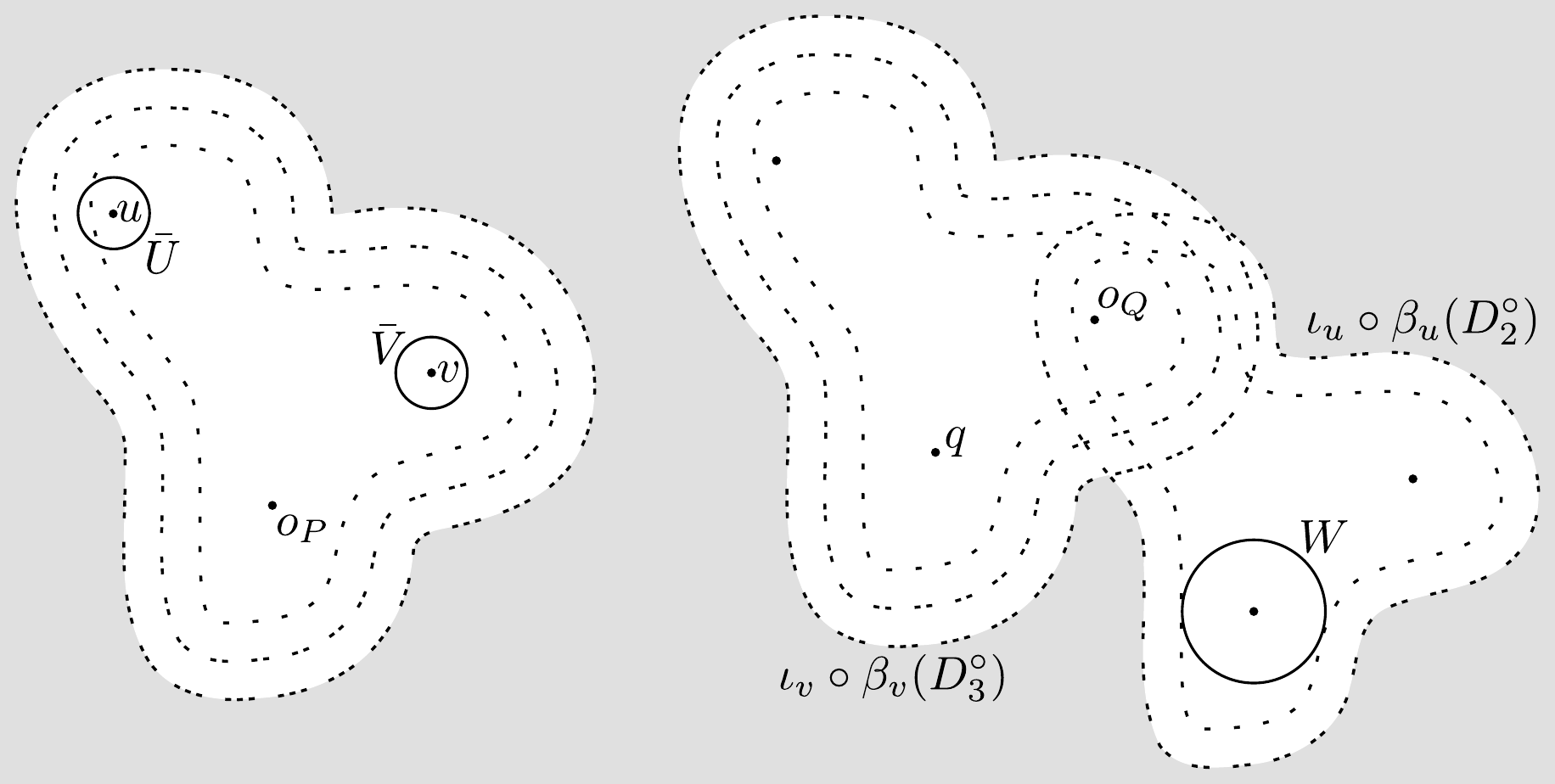}
\caption{The left side of the figure shows the three disk $D_1 \subset D_2 \subset D_3$ in $P$
and related objects. The right side
shows the fusion $Q=\beta_u(D_3^\circ) \fuse \beta_v(D_2^\circ)$.
}
\name{fig:embedding}
\end{figure}

We set $\sX$ to be the set defined at the end of the lemma. We will show that $\sX$ is open by proving that it is the same as the set 
$$\sY=\sM_{\imm}(D_3) \cap \Big(\sM_{\not \imm}(Q^q) \cup \sM_{-}(K^q,W^q)\Big).$$
Here, $K^q=\beta_q(K)$ and $W^q=\beta_q(W)$ represent subsets of $Q^q$, the surface $Q$ with basepoint moved to the position of $q$. 
Let $\sX^c=\sM_{\imm}(D_3) \smallsetminus \sX$ and $\sY^c=\sM_{\imm}(D_3) \smallsetminus \sY$. Since $\sX,\sY \subset \sM_{\imm}(D_3)$,
 is enough to prove that $\sX^c=\sY^c$. Observe that these sets are defined by:
$$\sX^c=\{S \in \sM~:~ \text{$\exists \iota:D_3 \imm S$ and $\iota(\bar U) \cap \iota(\bar V) \neq \emptyset$}\},$$
$$\sY^c=\{S \in \sM~:~ \text{$\exists \iota:D_3 \imm S$,
$\exists j:Q^q \imm S$, and $o_S \in j(W^q)$}\}.$$ 
(The equality for $\sY^c$ uses the fact that $W^q \subset K^q \subset Q^q$. In particular, the immersion $j:Q^q \imm S$ restricts to the immersion of $K^q$ into $S$.)

We will begin by showing that $\sX^c \subset \sY^c$. 
Let $S \in \sX^c$. Then there is an immersion $\iota:D_3 \imm S$
and there are points $u' \in \bar U$
and $v' \in \bar V$ so that $\iota(u')=\iota(v')$.
The balls $\bar U$ and $\bar V$ are canonically identified with a closed ball of radius $\epsilon$
in the plane. Using this identification, define the vectors $\bu=u'-u$ and $\bv=v'-v$. Since $u,v \in D_1$,
and $2 \epsilon<\ER(D_1 \subset D_2^\circ)$, there are ball embeddings
$$\BE_u:\bar B_{2 \epsilon} \to D_2^\circ \and \BE_v:\bar B_{2 \epsilon} \to D_2^\circ.$$
By construction, $u'=\BE_u(\bu)$ and $v'=\BE_v(\bv)$. Since, these points have the same image under
$\iota$, so do the points $u''=\BE_u(\bu-\bv)$ and $v$. Let $s=\iota(u'')=\iota(v)$. We get two immersions into $S^s$.
Since $D_3 \imm S$, we have $\beta_v(D_3)=D_3^v \imm S^s$. In addition, since
$\iota(u'')=s$, we have an immersion $k:\beta_{u''}(D_3) \imm S^s$.
Consider the following function:
$$f:D_2 \to D_3; \quad p \mapsto \BE_p(\bu-\bv).$$
The map $f$ respects the local translation structure and satisfies $f(u)=u''$. It follows that the following composition is an immersion:
$$k \circ \beta_{u''} \circ f \circ \beta_u^{-1}:\beta_u(D_2) \imm S^s.$$
In particular, we are using the structure the following commutative diagram:
\begin{equation}
\name{diagram:1}
\begin{tikzcd}[column sep=scriptsize, row sep=scriptsize]
D_2 \arrow{r}{f} \arrow{d}{\beta_u} & D_3 \arrow[squiggly]{r}{\iota} \arrow{d}{\beta_{u''}} & S \arrow{d}{\beta_{s}}\\
\beta_u(D_2) \arrow[squiggly]{r} & \beta_{u''}(D_3) \arrow[squiggly]{r}{k} & S^s
\end{tikzcd}
\end{equation}

By restricting to the interiors, we have immersions
$\beta_u(D_2^\circ)\imm S^s$ and $\beta_v(D_3^\circ) \imm S^s$. It follows from the Fusion Theorem that we can immerse the fusion $Q=\beta_u(D_2^\circ) \fuse \beta_v(D_3^\circ)$ into $S^s$.

We will now investigate properties of the immersion $j':Q \imm S^s$.
Recall that we set $q=\iota_v \circ \beta_v(o_P) \in Q$.
We claim that $j'(q)=\beta_s(o_S)$. This follows from the commutative diagram:
\begin{center}
\begin{tikzcd}[column sep=scriptsize, row sep=scriptsize]
D_3 \arrow[rightsquigarrow]{rr}{}{\iota} \arrow{d}{}{\beta_v} & & S \arrow{d}{}{\beta_s} \\
D_3^v \arrow[rightsquigarrow]{r}{}{\iota_v} & Q  \arrow[rightsquigarrow]{r}{}{j'} & S^s
\end{tikzcd}
\end{center}
It then follows that there is an immersion $j:Q^q \imm S$, and further immersions which make the following diagram commute:
\begin{center}
\begin{tikzcd}[column sep=scriptsize, row sep=scriptsize]
D_3 \arrow[rightsquigarrow]{rr}{}{\iota} \arrow{dd}{}{\beta_v} \arrow[squiggly]{dr} 
& & 
S \arrow{dd}{}{\beta_s} \\
& Q^q \arrow[squiggly]{ur}{}{j} & \\
D_3^v \arrow[rightsquigarrow]{r}{}{\iota_v} & Q  \arrow[rightsquigarrow]{r}{}{j'} \arrow{u}{}{\beta_q} & S^s
\end{tikzcd}
\end{center}

We make a similar argument to the above to show that $o_S \in j(W)$.
Observe that the immersion $\beta_u(D_2^\circ) \imm S^s$ factors
through $Q$. Modifying diagram \ref{diagram:1}, we have the following following commutative diagram:
$$\begin{tikzcd}[column sep=scriptsize, row sep=scriptsize]
D_2^\circ \arrow{r}{f} \arrow{dd}{\beta_u} & D_3 \arrow[squiggly]{r}{\iota} & S \arrow{dd}{\beta_{s}}\\
& Q^q \arrow[squiggly]{ur}{}{j} & \\
\beta_u(D_2^\circ) \arrow[squiggly]{r}{\iota_u} & Q \arrow[squiggly]{r}{j'} \arrow{u}{\beta_q} & S^s
\end{tikzcd}
$$
Note that the ball of radius $2 \epsilon$ centered at $o_P$
sits inside of $D_1^\circ$ and maps onto $W$ under the map $\iota_u \circ \beta_u$. Consider the point 
$r=\BE_{o_P}(\v-\u) \in D_1$. Since $r$ is within $2 \epsilon$
of $o_P$, we know that $\iota_u \circ \beta_u(r) \in W$.
This in turn is equivalent to the statement that
$\beta_q \circ \iota_u \circ \beta_u(r) \in W^q.$
It remains to show that $j \circ \beta_q \circ \iota_u \circ \beta_u(r) =o_S.$
This follows from commutativity of the diagram above
since $f(r)=o_P$.

We will now show that $\sY^c \subset \sX^c$. Choose $S \in \sY^c$. Then there is an immersion 
$\iota:D_3 \imm S$ and an immersion $j:Q^q \imm S$ so that $o_S \in j(W^q)$. 

Because we defined $s=\iota(v)$, there is an immersion $\iota':\beta_v(D_3) \imm S^s$. 
Since
$\iota_v:\beta_v(D_3^\circ) \imm Q$, and $\iota_v$ sends $\beta_v(o_P)$ to $q$,
we get an immersion $\iota_v':D_3^\circ \imm Q^q$. 
Moreover, this immersion satisfies $\iota_v'=\beta_q \circ \iota_v \circ \beta_v$.
In particular, $\iota_v'(v)=\beta_q(o_Q)$. 
By uniqueness of immersions, we get $\iota|_{D_3^\circ}=j \circ \iota_v'$.
The map $\beta_s \circ j \circ \beta_q$ sends $o_Q$ to the basepoint $\beta_s(s)$ of $S^s$,
so this map is an immersion, which we call $j'$. 
The situation is summarized by following commutative diagram:
\begin{center}
\begin{tikzcd}[column sep=scriptsize, row sep=scriptsize]
D_3^\circ \arrow[rightsquigarrow]{r}{}{\iota_v'} \arrow{d}{}{\beta_v} & Q^q \arrow[rightsquigarrow]{r}{j} & S \arrow{d}{}{\beta_s} \\
\beta_v(D_3^\circ) \arrow[rightsquigarrow]{r}{}{\iota_v} & Q \arrow{u}{\beta_q} \arrow[rightsquigarrow]{r}{}{j'} & S^s
\end{tikzcd}
\end{center}

By definition of $Q$, there is also an immersion $\iota_u:\beta_u(D_2^\circ) \imm Q$. We define
$\iota_u'=\beta_q \circ \iota_u \circ \beta_u.$ This is not an immersion, but respects the translation
structure. We have the following commutative diagram:
\begin{center}
\begin{tikzcd}[column sep=scriptsize, row sep=scriptsize]
D_2^\circ \arrow{r}{}{\iota_u'} \arrow{d}{}{\beta_u} & Q^q \arrow[rightsquigarrow]{r}{j} & S \arrow{d}{}{\beta_s} \\
\beta_u(D_2^\circ) \arrow[rightsquigarrow]{r}{}{\iota_u} & Q \arrow{u}{\beta_q} \arrow[rightsquigarrow]{r}{}{j'} & S^s
\end{tikzcd}
\end{center}
Since $o_S \in j(W^q)$, there is a $w^q \in W^q$ so that $j(w^q)=o_S$.
Let $w=\beta_q^{-1}(w_q) \in W$. By definition of $W$, there is a point $p \in P$ so that
$d(p,o_P) \leq 2 \epsilon$ and $\iota_u \circ \beta_u(p)=w$. 
Because $2 \epsilon<\ER(o_P \in D_1^\circ)$, we have $p \in D_1^\circ$
and $p=\BE_{o_P}(2\bw)$ for some vector $\bw \neq 0$ with $|\bw| \leq \epsilon$. 

Let $p_u=\BE_{o_P}(\bw)=\BE_p(-\bw) \in D_1$ and let $p_v=\BE_{o_{P}}(-\bw)$. Since
$j \circ \iota'_u(p)=o_S$ and $\iota(o_P)=o_S$, we must also have
$j \circ \iota_u'(p_u)=\iota(p_v)$. Consider the maps
$$f_u:D_1 \to D_2^\circ; \ast \mapsto \BE_\ast(\bw)  
\and
f_v:D_1 \to D_2^\circ; \ast \mapsto \BE_\ast(-\bw).$$
We claim that $j \circ \iota_u' \circ f_u=\iota\circ f_v$. This is because these maps
respect the local translation structure, and they agree at one point:
$$j \circ \iota_u' \circ f_u(o_P)=j \circ \iota_u'(p_u)=\iota(p_v)=\iota \circ f_v(o_P).$$
In particular, $j \circ \iota_u' \circ f_u(u)=\iota \circ f_v(u).$
Observe that $f_v(u) \in \bar U$. We will use this equation to find a point in $\bar V$ with the same image.
Observe that $\iota_u'(u)=\iota_v'(v)=\beta_q(o_Q)$. 
Because these maps respect the translation structure, 
we can write:
$$
\begin{array}{rcl}
j \circ \iota_u' \circ f_u(u) & = & j \circ \iota_u' \circ \BE_u(\bw)=
j \circ \BE_{\iota_u'(u)}(\bw)=j \circ \BE_{\iota_v'(v)}(\bw) \\
& =  &j \circ \iota_v' \circ \BE_{v }(\bw)=\iota \circ \BE_{v }(\bw).
\end{array}
$$
Here, we must be careful that whenever we commute $\BE(\bw)$ (say with $\iota_u'$
for the second equal sign), we must check the corresponding point ($u$ in this case) 
satisfies that the embedding radius into the domain of the subsequently applied map
($j$ in this case) is larger than $\epsilon$. So for the second equals sign, this involves noting
that $u \in D_1$, $\iota_u'$ is defined on $D_2^\circ$ and $\ER(D_1 \subset D_2^\circ)>\epsilon$.
The conclusion is the observation that 
$f_v(u) \in \bar U$, $\BE_{v }(\bw) \in \bar V$,
and $\iota \circ f_v(u)=\iota \circ \BE_{v }(\bw)$. This proves that
$\iota(\bar U) \cap \iota(\bar V) \neq \emptyset$, and therefore $S \in \sX^c$. 
\end{proof}

We now prove the embedding theorem:

\begin{proof}[Proof of Theorem \ref{thm:embedding}]
Let $D_1$ be a closed disk in a planar surface $P$. We will show that $\sM_{\emb}(D)$ is open. 
Let $T \in \sM_{\emb}(D_1)$ be a translation surface. Then $D_1 \emb T$
and we can find a $D_3 \in \cdisk(P)$
so that $D_1 \subset D_3^\circ$ and $D_3 \emb T$. 

For each pair $u,v \in D_1$, we will use the lemmas above to produce
a neighborhood $U(u,v)$ of $u$ and a neighborhood $V(u,v)$ of $v$ and
an open set $\sO(u,v) \subset \sM$. If $u \neq v$, we construct $U(u,v)$ and $V(u,v)$
as in Lemma \ref{lem:1}, and let $\sO(u,v) \subset \sM$ be the open set produced. 
If $u=v$, we take $U(u,v)=V(u,v)$ to be the set $U$ in Lemma \ref{lem:2}, and still define 
$\sO(u,v) \subset \sM$ to be the open set produced. 

The collection $\{U^\circ(u,v)\times V^\circ(u,v)~:~u,v \in D_1\}$ forms an open covering of the compact set $D_1 \times D_1$. So, we can find a finite set of pairs $(u_k,v_k)$ for $k=1,\ldots, K$ so that
$$D_1 \times D_1 \subset \bigcup_{k=1}^K U_k \times V_k,$$
where $U_k=U(u_k,v_k)$ and $V_k=V(u_k,v_k)$. We define the open set
$$\sU=\bigcap_{k=1}^K \sO_k, \quad \text{with $\sO_k=\sO(u_k,v_k)$}.$$

Two claims about $\sU$ will prove the theorem. First, we claim that $T \in \sU$.
This is because $D_3 \emb T$. It follows then from the statements of 
Lemma \ref{lem:1} and Lemma \ref{lem:2} that $T \in \sO(u,v)$ for all $u,v\in D_1$. 

Second, we claim that if $S \in \sU$, then $D_1 \emb S$. Note that the statement $S \in \sO(u,v)$ for any $u,v \in D_1$
implies that there is an immersion $\iota:D_3 \imm S$. We claim that the restriction $\iota|_{D_1}$ is an embedding.
Suppose otherwise. Then there are distinct $u,v \in D_1$ so that $\iota(u)=\iota(v)$. Because of our
covering, there is a $k$ so that $u \in U_k$ and $v \in V_k$. But the fact that $T \in \sO_k$ precludes
the possibility of $\iota(u)=\iota(v)$ by definition of $\sO_k$. 
\end{proof}

We now prove the disjointness theorem:

\begin{proof}[Proof of Theorem \ref{thm:disjointness}]
The proof is similar in spirit to the prior proof. Let $D$ be a closed disk, and let $K_1,K_2 \subset D$ be closed and disjoint.
We can assume without loss of generality that $K_1$ and $K_2$ lie in the interior of $D$. (Whenever a closed disk immerses in a surface,
there is a slightly larger closed disk which also immerses.)

Let $S \in \sM_\emptyset(D; K_1, K_2)$. Then, there is an immersion $\iota:D \imm S$ and $\iota(K_1) \cap \iota(K_2)=\emptyset$.
For each pair of points $(u,v) \in K_1 \times K_2$, we use Lemma \ref{lem:1} to choose closed balls 
$\bar U=\bar U(u,v)$ about $u$ and $\bar V=\bar V(u,v)$ about $v$ so that 
$\iota(\bar U) \cap \iota(\bar V)=\emptyset$ and so that the following set is open: 
$$\sO(u,v)=\{S \in \sM~:~\text{$\exists \iota:D \imm S$ and $\iota(\bar U) \cap \iota(\bar V)=\emptyset$}\}.$$
Then, the collection of interiors of sets in
$\{\bar U(u,v) \times \bar V(u,v)~:~(u,v) \in K_1 \times K_2\}$
covers $K_1 \times K_2$. By compactness, there is a finite collection of pairs
$\{(u_1,v_1), \ldots, (u_k,v_k)\} \subset K_1 \times K_2$ so that 
$$K_1 \times K_2 \subset \bigcup_{i=1}^k \bar U(u_i,v_i) \times \bar V(u_i,v_i).$$
We observe that
$$S \in \bigcap_{i=1}^k \sO(u_i,v_i) \subset \sM_\emptyset(D; K_1, K_2).$$
\end{proof}

\begin{proof}[Proof of Proposition \ref{prop:E-}]
Let $D$ be a closed disk in a translation surface, and let $K \subset D$ be a closed subset.
We will show that $\sE_-(D,K)$ is open in $\sE$. 

Choose $(S,s) \in \sE_-(D,K)$. Then, there is an immersion $\iota:D \imm S$ and $s \not \in \iota(K)$.
We can therefore choose a closed metric ball $B$ about $s$ which is isometric to a Euclidean ball
and disjoint from $\iota(K)$. By lifting, we also get an immersion $\tilde \iota:D \imm \tilde S$. 
Let $p_S:\tilde S \to S$ be the universal covering map.
Choose a $\tilde s \in \tilde S$ so that $p_S(\tilde s)=s$ and take $\tilde B \subset \tilde S$ to be the closed ball of the same radius
about $\tilde s$. Choose a $D' \in \cdisk(\tilde S)$ to be large enough so that $\tilde B \subset D'$ and $\tilde \iota(D) \subset D'$. 
We claim that
$$(S,s) \in \sE_+(D_2,\tilde B^\circ) \cap \pi^{-1}\big(\sM_\emptyset(D_2; \tilde \iota(K), \tilde B)\big) \subset \sE_-(D,K).$$
This will prove that $\sE_-(D,K)$ is open.

First, we prove that $(S,s)$ lies in this intersection. The restriction of the covering map $p_S$ gives an immersion
$D_2 \imm S$. Moreover $s \in B^\circ=p_S(\tilde B^\circ)$. This proves that $(S,s) \in \sE_+(D_2,\tilde B^\circ)$.
In addition, by uniqueness of immersions we have $\iota=p_S \circ \tilde \iota$. So, 
$p_S\big(\tilde \iota(K)\big)=\iota(K)$, and we constructed $B$ to be disjoint from $\iota(K)$. Thus,
$S \in \sM_\emptyset(D_2; \tilde \iota(K), \tilde B)$.

Second, we must show that we have the inclusion. Suppose $(T,t) \in \sE_+(D_2,\tilde B^\circ)$ and $T \in \sM_\emptyset(D_2; \tilde \iota(K), \tilde B)$.
Then, we get an immersion $j:D_2 \imm T$, $t \in j(\tilde B^\circ)$ and $j \circ \tilde \iota(K) \cap j(\tilde B)=\emptyset$. 
Since $t \in j(\tilde B)$, we know that $t \not \in j \circ \tilde \iota(K)$. Observe that $j \circ \tilde \iota$ is the unique immersion
of $D$ into $T$, so this proves that $(T,t) \in \sE_-(D,K)$.
\end{proof}

\section{Translation structures on the disk}
\name{sect:disk}

\subsection{Equivalence of topologies}
\name{sect:equivalence}
Recall that $\tilde M \subset \sM$ is the collection of all isomorphism classes
of translation structures on the disk, and that this space has been identified
with the collection of all planar surfaces, which are fibers in $\tE \subset \sM$.
The spaces $\tM$ and $\tE$ inherit subspace topologies from their inclusions into $\sM$ and $\sE$, respectively.
These spaces were also formally investigated in \cite{HooperImmersions1}, and were given a topology in that paper.
The paper \cite{HooperImmersions1} also placed topologies on $\tM$ and $\tE$. In this subsection,
we will prove that these topologies are the same.

\begin{lemma}
The subspace topology on $\tM$ induced by the immersive topology on $\sM$ is
the same as the topology on $\tM$ defined in \cite{HooperImmersions1}.
\end{lemma}

\begin{proof}
Recall that a topology is a collection of sets (satisfying certain axioms), which by definition are open.
For this proof, let $\sI$ denote the the topology on $\tM$ as defined in this paper,
and let $\sI'$ denote the topology on $\tM$ defined in \cite{HooperImmersions1}.

We will begin by explaining why $\sI' \subset \sI$.
The topology on $\sI'$ was defined in \cite{HooperImmersions1} to be the coarsest topology so that:
\begin{itemize}
\item Whenever $K$ is a closed disk in a planar surface, $\tM_{\imm}(K)=\{P \in \tM~:~K \imm P\}$ is open.
\item Whenever $U$ is an open disk in a planar surface, $\tM_{\not \emb}(U)=\{P \in \tM~:~U \not \emb P\}$ is open.
\end{itemize}
Each set of the form $\tM_{\imm}(K)$ is open in $\sI$, since $\tM_{\imm}(K)=\sM_{\imm}(K) \cap \tM$. 

Now let $P \in \tM_{\not \emb}(U)$. We will show that there is an open set in $\sI$ which contains
$P$ and is contained in $\tM_{\not \emb}(U)$. There are two possibilities. If $U \not \imm P$, then we have
$$P \in \tM \cap \sM_{\not \imm}(U) \subset \tM_{\not \emb}(U),$$
and we conclude by noting that $\tM \cap \sM_{\not \imm}(U) \in \sI$.
Otherwise, there is an immersion $\iota:U \imm P$. 
In this case, we take an increasing family $\{D_t \in \cdisk(U)\}_{t>0}$ so that 
$\bigcup_{t>0} D_t=U$. Since $U \not \emb P$, there is a $D \in \{D_t\}$ so that $D \not \emb P$. 
We claim that if $\iota(D)$ immerses in a planar surface $Q$, then $U \not \emb Q$. 
Assume to the contrary that there is an immersion $j:\iota(D) \imm Q$ and an embedding $e:U \emb Q$. 
By the uniqueness of immersions, we have
$e|_{D}=j \circ \iota|_D$. But, $\iota|_D$ is not an embedding so $e$ is also not. This proves that
\begin{equation}
\name{eq:P}
P \in \tM_{\imm}\big(\iota(D)\big) \subset \tM_{\not \emb}(U).
\end{equation}
But, it is not clear that $\tM_{\imm}\big(\iota(D)\big) \in \sI$, because
$\iota(D)$ need not be a disk. Let $R \in \tM_{\imm}\big(\iota(D)\big)$. Then there is an immersion
$k:\iota(D) \imm R$. By continuity $k \circ \iota(D)$ is compact, so we can find an $E \in \cdisk(R)$ so
that $k \circ \iota(D) \subset E$. Then, 
$$R \in \tM_{\imm}(E) \subset \tM_{\imm}\big(\iota(D)\big).$$
From the previous paragraph, we know $\tM_{\imm}(E) \in \sI$, which implies that $\tM_{\imm}\big(\iota(D)\big) \in \sI$. Then, it follows from equation \ref{eq:P}
that $\tM_{\imm}\big(\iota(D)\big) \in \sI$.

Now we will show that $\sI \subset \sI'$. As in the prior proof, we work through our subbasis of $\sI$.
This consists of intersections with $\tM$ of the four types of sets described in \S \ref{sect:top M} for the subbasis for topology on $\sM$. We work through these four types one at a time. First, let $D$ be a closed disk in a translation surface. Then $D$ has a lift $\tilde D$ to its universal cover by Proposition \ref{prop:lift}. From above,
and because $D$ and $\tilde D$ are indistinguishable via immersions (Proposition \ref{prop:indistinguishable}),
we have 
$$\tM \cap \sM_{\imm}(D)=\tM \cap \sM_{\imm}(\tilde D)=\tM_{\imm}(\tilde D).$$
Second, let $U$ be an open disk in a translation surface. Again, let $\tilde U$ be the lift to the surface's
universal cover. We have 
$$\tM \cap \sM_{\not \imm}(U)=\tM \cap \sM_{\not \imm}(\tilde U)=\{P \in \tM~:~\tilde U \not \imm P\}.$$
This last set is denoted $\tM_{\not \imm}(\tilde U)$ in \cite{HooperImmersions1},
and this set is open by \cite[Theorem \ref*{I-thm:open}]{HooperImmersions1}.

To handle the remaining two basis elements of $\sI$, we must recall two facts from \cite{HooperImmersions1}.
First, the canonical section of the projection $\tpi:\tE \to \tM$ is given by
$$\tilde \sigma:\tM \to \tE; \quad P \mapsto (P,o_P)$$
is continuous by \cite[Proposition \ref*{I-prop:section}]{HooperImmersions1}. Second, for any closed disk $\tilde D$ in a planar surface, any $\tilde U \subset \tilde D^\circ$ open, and any $\tilde K \subset \tilde D$ closed, 
the following two sets are open:
\begin{equation}
\name{eq:tE+}
\tE_+(\tilde D,\tilde U)=\{(P,p) \in \tE~:~\text{$\exists \iota:\tilde D \imm P$ and $p \in \iota(\tilde U)$}\}.
\end{equation}
\begin{equation}
\name{eq:tE-}
\tE_-(\tilde D,\tilde K)=\{(P,p) \in \tE~:~\text{$\exists \iota:\tilde D \imm P$ and $p \not \in \iota(\tilde K)$}\}.
\end{equation}
Sets of the form $\tE_+(\tilde D,\tilde U)$ form the subbasis for the topology on $\tE$ defined in 
 \cite[\S \ref*{I-sect:topology total}]{HooperImmersions1}, and sets of the form $\tE_-(\tilde D,\tilde K)$
are open by \cite[Proposition \ref*{I-prop:E-}]{HooperImmersions1}.

Now we consider the third and fourth types of open sets for the topology $\sI$. Let $D$ be a closed
disk in a translation surface, let $\tilde U \subset D^\circ$ be open, and let $K \subset D$ be closed.
Let $\ell_D:D \emb \tilde D$ be the lift of $D$ to the universal cover of the translation surface containing $D$ 
as in Proposition \ref{prop:lift}.
Let $\tilde U=\ell_D(U)$ and $\tilde K=\ell_D(K)$. Observe that we have the following identities:
$$\tM \cap \sM_+(D,U)=\tM \cap \sM_+(\tilde D,\tilde U)=\tilde \sigma^{-1} \big(\tE_+(\tilde D,\tilde U)\big).$$
$$\tM \cap \sM_-(D,K)=\tM \cap \sM_-(\tilde D,\tilde K)=\tilde \sigma^{-1} \big(\tE_-(\tilde D,\tilde K)\big).$$
It follows that both sets are open in $\sI'$ by the continuity of $\tilde \sigma$. 
\end{proof}

\begin{lemma}
The subspace topology on $\tE$ induced by the immersive topology on $\sE$ is
the same as the topology on $\tE$ defined in \cite{HooperImmersions1}.
\end{lemma}
\begin{proof}
For this proof, let $\sJ$ denote the the topology on $\tE$ given as the subspace
topology inherited from the immersive topology as defined in this paper,
and let $\sJ'$ denote the topology on $\tE$ defined in \cite{HooperImmersions1}.

We will begin by explaining why $\sJ' \subset \sJ$. The topology is defined in
$\sJ'$ is the coarsest topology so that $\tilde \pi:\tE \to \tM$ is continuous, and so that
sets of the form 
$\tE_+(\tilde D,\tilde U)$, defined in equation \ref{eq:tE+} are open. Recall that the projection
$\pi:\sE \to \sM$ is defined to be continuous in the topology on $\sE$, and 
we have $\tpi=\pi|_{\tE}$. Thus, $\tpi$ is continuous in the topology $\sJ$. 
Also observe that for any closed disk $\tilde D$ in a planar surface, and any $\tilde U \subset \tilde D^\circ$ open,
we have 
$$\tE_+(\tilde D,\tilde U)=\tE \cap \sE_+(\tilde D, \tilde U).$$
The set $\sE_+(\tilde D, \tilde U)$ is a subbasis element of $\sJ$. 

Now we will show that $\sJ \subset \sJ'$. By definition $\sJ$ is the coarsest topology 
so that $\pi|_{\tE}$ is continuous, and for any closed disk $D$ in a translation surface
any $U \subset D^\circ$ open, and $K \subset D$ closed, we have that $\tE \cap \sE_+(D, U)$. 
Again $\pi|_{\tE}=\tpi$, which is defined to be continuous in $\sJ'$. 
Fix $D$, $U$ and $K$ as above. 
Define $\tilde D$, $\tilde U$ and $\tilde K$ as in the last paragraph of the previous proof. Then, we have:
$$\tE \cap \sE_+(D, U)=\tE \cap \sE_+(\tilde D, \tilde U)=\tE_+(\tilde D, \tilde U).$$
Here, the set on the right hand side is as in equation \ref{eq:tE+},
and are in $\sJ'$ as discussed surrounding that equation.
\end{proof}

\subsection{Continuity and the universal cover}
Since the subspace topology defined on $\tM \subset \sM$ is the same as the topology
defined in \cite{HooperImmersions1}, we note that this topology has a particularly nice subbasis.
We restate Corollary \ref*{I-cor:immersion subbasis} of \cite{HooperImmersions1}:
\begin{corollary}[Subbasis of $\tM$]
\name{cor:subbasis M}
A subbasis for the topology on $\tM \subset \sM$ is given by sets of the following two forms:
\begin{itemize}
\item Sets of the form $\tM_{\imm}(K)=\{Q \in \tM~:~K \imm Q\}$ where $K \in \cdisk(P)$ for some $P \in \tM$.
\item Sets of the form $\tM_{\not \imm}(U)=\{Q \in \tM~:~U \not \imm Q\}$ where $U \in \disk(P)$ for some $P \in \tM$.
\end{itemize}
\end{corollary}

Let $\upsilon:\sM \to \tM$ be the map which sends a translation surface $S$ to its universal cover $\tilde S$.
We show this map is continuous as claimed in Theorem \ref{thm:universal cover}.

\begin{proof}[Proof of Theorem \ref{thm:universal cover}]
We will check that the preimage of the subbasis of open sets provided by Corollary \ref{cor:subbasis M}
is open. By Proposition \ref{prop:lifted immersion},
the statement that a simply connected set immerses in a surface $S$ is logically equivalent to
the statement that the set immerses in $\tilde S$. Thus,
$\upsilon^{-1}\big(\tM_{\imm}(K)\big)=\sM_{\imm}(K)$ for all closed disks $K$, and 
$\upsilon^{-1}\big(\tM_{\not \imm}(U)\big)=\sM_{\not \imm}(U)$ for all open disks $U$.
These sets are open by definition of the immersive topology on $\sM$.
\end{proof}

\subsection{The Hausdorff property}
We will now prove that the immersive topologies on $\sM$ and $\sE$ are Hausdorff:
\begin{proof}[Proof of Theorem \ref{thm:immersive}]
We proved that the topologies on $\sM$ and $\sE$ are second countable in \S \ref{sect:countability}. It remains
to prove that the topologies are Hausdorff.

We begin by showing $\sM$ is a Hausdorff space. 
Let $S$ and $T$ be distinct points in $\sM$. We will prove that they can be separated by disjoint open sets. 

First, suppose that their universal covers $\tilde S$ and $\tilde T$ are distinct. By Theorem \cite[Theorem \ref*{I-thm:Hausdorff}]{HooperImmersions1}, the topology on $\tM$ is Hausdorff. So, there are disjoint open sets $\sU$ and 
$\sV$ in $\tM$ so that $\tilde S \in \sU$ and $\tilde T \in \sV$. By Theorem \ref{thm:universal cover},
the sets $\upsilon^{-1}(\sU)$ and $\upsilon^{-1}(\sV)$ are open. They separate $S$ and $T$.

Now suppose that $\tilde S=\tilde T$. Let $p_S:\tilde S \to S$ and $p_T: \tilde S \to T$ be the associated covering maps. By distinctness, we can assume without loss of generality that there is a $\tilde s \in \tilde S$ so
that $p_S(\tilde s)=o_S$ but $p_T(\tilde s) \neq o_T$. Choose a $D \in \cdisk(\tilde S)$
so that $\tilde s \in D^\circ$. By discreteness of the lifts of the basepoint of $T$ to $\tilde S$,
we can choose an open set $U \subset D^\circ$ so that $\tilde s \in U$ and $\bar U \cap p_T^{-1}(o_T)=\emptyset$. 
Then we have $S \in \sM_+(D,U)$ and $T \in \sM_-(D,\bar U)$. These sets are open by definition
of the immersive topology on $\sM$. Moreover, they are disjoint since
$R \in \sM_-(D,\bar U)$ implies that there is an immersion $\iota:D \to R$ with $o_R \not \in \iota(\bar U)$.
In this case, $o_R \not \in \iota(U)$ so $R \not \in \sM_+(D,U)$.

Now we will prove that the immersive topology on $\sE$ is Hausdorff.
Let $(S,s)$ and $(T,t)$ be distinct points in $\sE$. If $S \neq T$, then since $\sM$ is Hausdorff,
we can find disjoint open sets $\sU, \sV \subset \sM$ so that $S \in \sU$ and $T \in \sV$. Then,
$(S,s) \in \pi^{-1}(\sU)$ and $(T,t) \in \pi^{-1}(\sV)$. These sets are disjoint. The definition of the topology on $\sE$ guarantees that $\pi$ is continuous, so these sets are also open.

Now suppose $S=T$. Then, $s,t \in S$ are distinct. Choose lifts $\tilde s, \tilde t \in \tilde S$.
Then we can find a closed disk $D \in \cdisk(\tilde S)$ so that $\tilde s, \tilde t \in D^\circ$. 
Then, the open set $p_S(D^\circ) \subset S$ contains both $s$ and $t$. Since $p_S(D^\circ)$ is an open set in a surface, it is Hausdorff, so we can find an open set $U$ containing $s$ so that $t \not \in \bar U$.
Then, $(S,s) \in \sE_+(D,U)$ and $(S,t) \in \sE_-(D,\bar U)$. These sets are disjoint by the same reasoning
as we used to conclude that $\sM$ is Hausdorff.
\end{proof}

\subsection{Proof of the Projection Theorem}

\begin{proof}[Proof of Theorem \ref{thm:covering projection}]
We will prove the covering projection, $p:(S,\tilde s) \mapsto p_S(\tilde s) \in \sE$ is continuous.
Since $\pi \circ p:(S,\tilde s) \mapsto S$ is continuous, it suffices to show that $p^{-1}\big(\sE_+(D,U)\big)$ is open for every
closed disk $D$ and every open $U \subset D^\circ$. 

Suppose that $(S,\tilde s) \in p^{-1}\big(\sE_+(D,U)\big)$. Let $s=p(S,\tilde s)=p_S(\tilde s)$. Then,
there is an immersion $\iota:D \imm S$ and $s \in \iota(U)$. It follows that there is an immersion 
$\tilde \iota:D \imm \tilde S$, and there is a lift $\tilde s_\ast$ of $s$ inside of $\tilde \iota(U)$. 
Since $p_S(\tilde s)=p_S(\tilde s_\ast)$, there is an element of the deck group of the covering,
$\gamma:\tilde S \to \tilde S$,
so that $\gamma(\tilde s_\ast)=\tilde s$. Let $\tilde o_\ast=\gamma(\tilde o_S)$ be the image
of the basepoint of $\tilde S$. 

Choose a $\tilde D \in \cdisk(\tilde S)$ so that $\tilde s \in \tilde D^\circ$, $\tilde o_\ast \in \tilde D^\circ$, and $\gamma \circ \tilde \iota(D) \subset \tilde D^\circ$. 
Choose
$$\epsilon<\min \big\{\ER(\tilde s_\ast \in \tilde D^\circ), \frac{1}{2} \ER\big(\tilde s_\ast \in \tilde \iota(U)\big), \ER\big(\gamma \circ \tilde \iota(D) \subset \tilde D^\circ\big)\big\}.$$
Let $\tilde B_o$ and $\tilde B_s$ denote the open balls in $\tilde S$ of radius $\epsilon$ about $\tilde o_\ast$ and $\tilde s$, respectively. Observe that $S \in \sM_+(\tilde D, \tilde B_o)$ and 
$\tilde s \in \sE_+(\tilde D, \tilde B_s)$. 

We claim that if $(T,\tilde t)$ satisfies
$T \in \sM_+(\tilde D, \tilde B_o)$ and $\tilde t \in \tilde T \cap \sE_+(\tilde D, \tilde B_s)$,
then $t=p_T(\tilde t) \in \sE_+(D,U)$. Once this is proved, we see that 
$\sM_+(\tilde D, \tilde B_o) \times \sE_+(\tilde D, \tilde B_s)$
intersected with the domain of the projection is an open subset of the domain
which contains $(S,\tilde s)$ and is contained in $p^{-1}\big(\sE_+(D,U)\big)$.

Now we will prove our claim.
Since $\tilde t \in \tilde T \cap \sE_+(\tilde D, \tilde B_s)$,
there is an immersion $\tilde j:\tilde D \imm \tilde T$ and
$\tilde t \in \tilde j(\tilde B_s)$. By postcomposing with the covering map
$p_T$, we also get an immersion $j=p_T \circ \tilde j:\tilde D \imm T$,
and $o_T \in j(\tilde B_o)$ since $T \in \sM_+(\tilde D, \tilde B_o)$.
This means we can find a point $\tilde o_+ \in \tilde j(\tilde B_o) \cap p_T^{-1}(o_T)$. 
Then, there is a vector $\v$ with norm less than $\epsilon$ so that 
$$\tilde o_+=\BE_{\tilde j(\tilde o_\ast)}(\v).$$
Define the function 
$$f: \quad \gamma \circ \tilde \iota(D) \to \tilde D; \quad x \mapsto \BE_x(\v).$$
We claim that $j \circ f \circ \gamma \circ \tilde \iota:D \to T$
is an immersion. To see this note that each map locally respects the translation structure,
and the composition respects the basepoints,
$$j \circ f \circ \gamma \circ \tilde \iota(o_D)=p_T \circ \tilde j \circ f \circ \gamma \circ \tilde \iota(o_D)=p_T(\tilde o_+)=o_T.$$
Since $\tilde t \in \sE_+(\tilde D, \tilde B_s)$, we know that $\tilde t$ is within $\epsilon$
of $\tilde j(\tilde s)$. We also have $\tilde j(\tilde s)=\tilde j \circ \gamma(\tilde s_\ast)$. Therefore, $\tilde j(\tilde s)$ is within $\epsilon$ of 
$\tilde j \circ f \circ \gamma(\tilde s_\ast)$. So, by the triangle inequality,
$\tilde t$ is within $2 \epsilon$ of $\tilde j \circ f \circ \gamma(\tilde s_\ast)$.
Since $2 \epsilon<\ER\big(\tilde s_\ast \in \tilde \iota(U)\big)$, we conclude
that $\tilde t \in \tilde j \circ f \circ \gamma \circ \iota(U)$.
Therefore, $t \in j \circ f \circ \gamma \circ \iota(U)$,
which proves that $t \in \sE_+(D,U)$ as desired.
\end{proof}

The joint continuity of immersions is a consequence of Theorem \ref{thm:covering projection} and work in \cite{HooperImmersions1}.

\begin{proof}[Proof of Proposition \ref{prop:continuity immersions}]
Fix an open disk $U$. Let $\sI(U) \subset \sM$ be the set of $S \in \sM$ so that there is an immersion $\iota_S:U \imm S$. 
Recall that $I_U:\sI(U) \times U \to \sE$ is defined by $I_U(S,u)=\iota_S(u)$.
Let $\tilde \sI(U)=\sI(U) \cap \tM$, and let $\tilde I_U:\tilde \sI(U) \times U \to \tE$ be the restriction
of $I_U$ to $\tilde \sI(U)$. 
Proposition \ref*{I-prop:Continuity of immersions} of \cite{HooperImmersions1} states that for any open disk $U$,
the map $\tilde I_U$ is continuous. The continuity of $I_U$ then follows, because whenever $S \in \sI(U)$, we have $\tilde S \in \tilde \sI(U)$ and
$I_U(S,u)=p_S \circ \tilde I_U(\tilde S,u),$
where $p_S: \tilde S \to S$ is the covering map. Here, we are using both the continuity of the map $S \mapsto \tilde S$ and of the covering projection
(Theorem \ref{thm:covering projection}).
\end{proof}

\subsection{Convergence results}
Before proving our convergence results, we provide a lemma which produces a quotient
translation surface from the set of lifts of the basepoint of the surface to its universal cover.

\begin{lemma}
\name{lem:convergence}
Let $\tilde S \in \tM$ and suppose $\tilde O \subset \tilde S$ satisfies the following statements:
\begin{enumerate}
\item The set $\tilde O$ is discrete as a subset of $\tilde S$.
\item The collection $\Gamma=\{\beta_{\tilde o}:\tilde o \in \tilde O \}$ of basepoint changing isomorphisms forms a group of translation automorphisms of $\tilde S$.
\end{enumerate}
Then, there is a surface $S \in \sM$ with universal cover $\tilde S$ so that $p_S^{-1}(o_S)=\tilde O$. 
\end{lemma}

In the proof, we will use the fact that the basepoint changing map and basepoint changing isomorphism
are continuous when restricted to $\tM$ and $\tE$. The following is a restatement of 
Theorem \ref{I-thm:basepoint change} of \cite{HooperImmersions1}.

\begin{theorem}
\name{thm:basepoint disk}
The restriction of the basepoint changing map,
$\BC|_{\tE}:\tE \to \tM$ is continuous. The basepoint changing isomorphism
$q \mapsto \beta_p(q)$ is jointly continuous when restricted to 
$$\{(p,q) \in \tE^2~:~\pi(p)=\pi(q)\}.$$ 
The inverse basepoint changing isomorphism,
$r \mapsto \beta_p^{-1}(r)$ is continuous when restricted to
$$\{(p,r) \in \tE^2~:~\pi(r)= \BC(p)\}.$$
\end{theorem}

\begin{proof}[Proof of Lemma \ref{lem:convergence}]
Our surface $S$ will be the quotient $\tilde S/\Gamma$. First assume that this quotient is a translation surface $S$.
Let $\tilde o_S$ be the basepoint of $\tilde S$. 
By statement (2), the maps $\beta_{\tilde o}$ are all translation automorphisms.
Thus, they are of the form $\beta_{\tilde o}:\tilde S \to \tilde S$. 
Because the identity automorphism must be in $\Gamma$, we know that the base point of $\tilde S$, $\tilde o_S$,
lies in $\tilde O$. Also observe by definition of $\beta$ that $\beta_{\tilde o}(\tilde o)=\tilde o_S$. 
This proves the last remark that $p_S^{-1}(o_S)=\tilde O$, assuming that the quotient is a translation surface.

To prove that the quotient is a translation surface, we will show that each point $\tilde s \in \tilde S$
lies in an open set $U \subset \tilde S$ so that whenever $\beta_1, \beta_2 \in \Gamma$ satisfy $\beta_1(U) \cap \beta_2(U) \neq \emptyset$, we have 
$\beta_1=\beta_2$. Assume that $\tilde s$ does not have this property. Let $B_\epsilon$
denote the open $\epsilon$ ball about $\tilde s$. Then, for each $\epsilon$, there is a distinct pair $\beta_1, \beta_2 \in \Gamma$ so that $\beta_1(B_\epsilon) \cap \beta_2(B_\epsilon) \neq \emptyset$.
In other words, $B_\epsilon \cap \beta_1^{-1} \circ \beta_2(B_\epsilon) \neq \emptyset$.
In particular, the distance from $\tilde s$ to $\beta_1^{-1} \circ \beta_2(\tilde s)$ is less than $2 \epsilon$. 
By statement (2), $\beta_1^{-1} \circ \beta_2 \in \Gamma$. So there is an $\tilde o=\tilde o(\epsilon) \in \tilde O$
so that $\beta_1^{-1} \circ \beta_2 =\beta_{\tilde o}$. Applying this for the sequence $\epsilon=\frac{1}{n}$
produces a sequence $\tilde o_n \in \tilde O$ so that 
$$\lim_{n \to \infty} \beta_{\tilde o_n}(\tilde s)=\tilde s.$$
Let $\tilde s_n=\beta_{\tilde o_n}(\tilde s)$. 
Then, $\beta_{\tilde o_n}$ is a translation automorphism of $\tilde S$ which carries $\tilde s$ to $\tilde s_n$. It follows that
$\BC(\tilde S, \tilde s)=\BC(\tilde S, \tilde s_n)$. Let $\tilde T$ denote this common surface, and consider the maps
$$\beta_{\tilde s}:\tilde S \to \tilde T \and \beta_{\tilde s_n}:\tilde S \to \tilde T.$$
By definition, the first map carries $\tilde s$ to the basepoint of $\tilde T$, and the second carries 
$\tilde s_n$ to the basepoint of $\tilde T$. We claim that 
$$\beta_{\tilde o_n}=\beta_{\tilde s_n}^{-1} \circ \beta_{\tilde s}.$$
By the above remarks, both send $\tilde s$ to $\tilde s_n$. Since there can be only one translation automorphism which does this, they must be equal. We can recover the sequence $\langle \tilde o_n \rangle$ as
$$\tilde o_n=\beta_{\tilde o_n}^{-1}(\tilde o_S)=\beta_{\tilde s}^{-1} \circ \beta_{\tilde s_n}(\tilde o_S).$$
By Theorem \ref{thm:basepoint disk}, we have
$$\lim_{n \to \infty} \tilde o_n=\lim_{n \to \infty} \beta_{\tilde s}^{-1} \circ \beta_{\tilde s_n}(\tilde o_S)=
\beta_{\tilde s}^{-1} \circ \beta_{\tilde s}(\tilde o_S)=\tilde o_S.$$
Since $\beta_{\tilde o_n}$ is never the identity translation automorphism, the existence of the sequence $\langle \tilde o_n \in \tilde O\rangle$ approaching $\tilde o_S$ violates the discreteness of $\tilde O$. (Note that
$\tilde o_S$ is necessarily in $\tilde O$, because the group $\Gamma$ needs an identity element.)
\end{proof}

We now prove our convergence criterion for $\sM$.
\begin{proof}[Proof of Theorem \ref{thm:convergence in M}]
Let $\langle S_n \in \sM\rangle$ be a sequence of translation surfaces converging to $S$.
Let $o \in S$ and $o_n \in S_n$ denote basepoints, and let $p:\tilde S \to S$ and $p_n:\tilde S_n \to S_n$ denote the universal covering maps. 
We will prove that statements (1) to (2) of the theorem hold for $\tilde O=p^{-1}(o)$,

Recall that the map $\upsilon: \sM \to \tM$, which sends a surface to its universal cover, is continuous by Theorem \ref{thm:universal cover}.
Thus $\langle \tilde S_n \rangle$ converges to $\tilde S$.

Choose $\tilde o \in p^{-1}(o)$. To prove (1), we must find a sequence $\tilde o_n \in p_n^{-1}(o_n)$ so that $\langle \tilde o_n \rangle$ converges to $\tilde o$ in $\tE$. Choose a closed disk $D \in \cdisk(\tilde S)$ so that $\tilde o \in D^\circ$. Since $\langle \tilde S_n \rangle$ converges to $\tilde S$,
there is an $N$ so that for $n>N$, there is an immersion $\tilde \iota_n:D \imm \tilde S_n$. Note that $p_n \circ \tilde \iota_n:D \imm S$ is an immersion.
Choose an $\epsilon>0$ small enough
so that the open ball $B_\epsilon$ of radius $\epsilon$ about $\tilde o$ is contained in $D$. Since $\sM_+(D,B_\epsilon)$ is open and contains $S$,
there is an $M>N$ so that $o_n \in p_n \circ \tilde \iota_n(B_\epsilon)$ for $n>M$. It follows that for $n>M$, there is a point in $B_\epsilon$ so that 
$o_n$ is the image of the point under $p_n \circ \tilde \iota_n$. Now let $\epsilon \to 0$. We can choose a closest lift $p_n \in D^\circ$
to $\tilde o$ so that $p_n \circ \tilde \iota_n(p_n)=o_n$ and so that the sequence $\langle p_n \rangle$ converges to $\tilde o$ inside of $D^\circ$.
It then follows from joint continuity of immersions (Proposition \ref{prop:continuity immersions}) that $\tilde o_n=\tilde \iota_n(p_n)$
converges to $\tilde o$ inside of $\tE$.

Let $\langle n_k \rangle$ be an increasing
sequence of positive integers.
Let $\langle \tilde o_{n_k} \in \tilde S_{n_k} \rangle$ be a sequence which converges to $\tilde o \in \tE$ and satisfies $p_{n_k}(\tilde o_{n_k})=o_{n_k}$.
To show (2), we must prove that $p(\tilde o)$ is the basepoint $o \in S$. Since $\langle \tilde S_{n_k} \rangle$ converges to $\tilde S$,
we know by continuity of $\pi:\sE \to \sM$ that $\tilde o \in \tilde S$. We also know that $\langle S_n \rangle$ converges to $S$.
Therefore, by the Projection Theorem $\langle o_{n_k}=p_{n_k}(\tilde o_{n_k})\rangle$ tends to $p(\tilde o)$. Since each $o_{n_k}$ is the basepoint of $S_{n_k}$,
the limit $p(\tilde o)$ must be the basepoint of $S$. (This statement about basepoints follows from Proposition \ref{prop:continuity immersions}, for instance.)

Now we will prove the converse.
Suppose that $\langle S_n \in \sM \rangle$ is a sequence of surfaces.
Let $\langle \tilde S_n \in \tM \rangle$ be the sequence of universal covers, and
let $\tilde S \in \tM$ be another surface in $\tM$. Suppose $\tilde O \subset \tM$ is a discrete subset
satisfying the statements (1) and (2) of the theorem. We will check that this set $\tilde O$ satisfies the two statements of Lemma \ref{lem:convergence}. Observe that the first statement (that $\tilde O$ is discrete),
is tautologically satisfied.

We will now prove the second statement of Lemma \ref{lem:convergence}. 
Pick a $\tilde q \in \tilde O$. We will first show that $\beta_{\tilde q}$ is an automorphism of $\tilde S$.
This requires showing that $\BC(\tilde S, \tilde q)=\tilde S$. 
By statement (1) of the theorem, there is a sequence $\langle \tilde q_n \in p_n^{-1}(o_n) \rangle$
which converges to $\tilde q$. Then by continuity of $\BC$, observe that
$$\BC(\tilde S, \tilde q)=\lim_{n \to \infty} \BC(\tilde S_n, \tilde q_n)=\lim_{n \to \infty} \tilde S_n=\tilde S.$$

It remains to show that $\Gamma=\{\beta_{\tilde q}~:~\tilde q \in \tilde O\}$ forms a group.
First observe that $\Gamma$ contains the identity element, because the basepoint $\tilde o$ of $\tilde S$
lies in $\tilde O$. Observe that the basepoint $\tilde o_n$
of $\tilde S_n$ lies in $p_n^{-1}(o_n)$ for every $n$. Since $\langle \tilde S_n \rangle$ tends to $\tilde S$,
$\langle \tilde o_n \rangle$ tends to $\tilde o$. So by statement (2) of the Theorem, $\tilde o \in \tilde O$. 

Since the identity element lies in $\Gamma$, it suffices to prove that for any pair of elements
$\gamma_1,\gamma_2 \in \Gamma$, we also have $(\gamma_2 \gamma_1)^{-1} \in \Gamma$. We will use the fact
that 
$$\Gamma_n=\{\beta_{\tilde q_n}~:~\tilde q_n \in p_n^{-1}(o_n)\}$$
has this property since it is the deck group of the cover $p_n:\tilde S_n \to S_n$. 
Choose two elements $\tilde q, \tilde r \in \tilde O$. We will take $\gamma_1=\beta_{\tilde q}$ and $\gamma_2=\beta_{\tilde r}$. 
By statement (1) of the theorem, there are sequences $\langle \tilde q_n \in p_n^{-1}(o_n) \rangle$
and $\langle \tilde r_n \in p_n^{-1}(o_n) \rangle$ which converge to $\tilde q$ and $\tilde r$, respectively.
By Theorem \ref{thm:basepoint disk}, we know that 
\begin{equation}
\name{eq:r}
\beta_{\tilde r} \circ \beta_{\tilde q}(\tilde o)=\lim_{n \to \infty} \beta_{\tilde r_n} \circ \beta_{\tilde q_n}(\tilde o_n),
\end{equation}
where $\tilde o$ and $\tilde o_n$ denote the basepoints of $\tilde S$ and $\tilde S_n$, respectively.
Define 
$$\tilde s=\beta_{\tilde r} \circ \beta_{\tilde q}(\tilde o) \and
\tilde s_n=\beta_{\tilde r_n} \circ \beta_{\tilde q_n}(\tilde o_n).$$
Observe that by definition of $\beta$, we have 
$$\beta_{\tilde s}(\tilde s)=\tilde o \and \beta_{\tilde s_n}(\tilde s_n)=\tilde o_n.$$
A translation automorphism is determined by the image of a single point. So, we can conclude that
$$\beta_{\tilde s}=(\beta_{\tilde r} \circ \beta_{\tilde q})^{-1} \and
\beta_{\tilde s_n}=(\beta_{\tilde r_n} \circ \beta_{\tilde q_n})^{-1}.$$
So, it suffices to prove that $\tilde s \in \tilde O$. Since $\Gamma_n$ is a group,
$\beta_{\tilde s_n} \in \Gamma$. It follows that $\tilde s_n \in p_n^{-1}(o_n)$. 
Equation \ref{eq:r} showed that $\langle \tilde s_n \rangle$ converges to $\tilde s$,
so by statement (2) of the theorem, we know that $\tilde s \in \tilde O$. 
This concludes the proof that $\Gamma$ is a group.
\end{proof}

We now prove our criterion for convergence in $\sE$. 
\begin{proof}[Proof of Theorem \ref{thm:convergence in E}]
First suppose that $\langle (S_n,s_n) \in \sE \rangle$ converges to $(S,s) \in \sE$.
Since $\pi:\sE \to \sM$ is continuous, $\langle S_n \rangle$ converges to $S$. 
Now let $p:\tilde S \to S$ be the universal covering map.
Choose an arbitrary $\tilde s \in p^{-1}(s)$ and choose a $D \in \cdisk(\tilde S)$ so that
$\tilde s \in D^\circ$.
Let $\epsilon>0$ be small enough so that the open ball $B_\epsilon \subset \tilde S$ of radius
$\epsilon$ about $\tilde s$ lies in $D^\circ$. Observe that $(S,s) \in \sE_+(D,B_\epsilon)$. 
Since this set is open, there is an $N=N(\epsilon)$ so that $(S_n,s_n) \in \sE_+(D,B_\epsilon)$ for $n>N$.  
That is, there is an immersion $\iota_n:D \imm S_n$ and $s_n \in \iota_n(B_\epsilon)$. 
In other words, $B_\epsilon \cap \iota_n^{-1}(s_n) \neq \emptyset$ for $n>N(\epsilon)$. 
Let $q_n$ be a choice of closest point in $\iota_n^{-1}(s_n)$ to $\tilde s$
for $n>N(1)$. By letting $\epsilon$ tend to zero in the remarks above, we see that $q_n$
converges to $\tilde s$ within $D^\circ$. The immersions $\iota_n$ lift to immersions 
$\tilde \iota_n:D \imm \tilde S_n$.
Let $\tilde s_n=\tilde \iota_n(q_n)$ for $n>N(1)$. We have $\iota_n=p_n \circ \tilde \iota_n$, where $p_n:\tilde S_n \to S_n$ is the covering map. Thus, $p_n(\tilde s_n)=s_n$. By continuity of immersions applied
to the immersions $\iota_n:D \imm S_n$ converging to the restriction
of the identity on $D \subset S$, we see that $\tilde s_n=\tilde \iota_n(q_n)$ tends to $\tilde s$.

Now suppose statements (1) and (2) hold for the sequence $\langle (S_n,s_n) \in \sE \rangle$ 
and $(S,s) \in \sE$. We will prove that the sequence converges to $(S,s)$. 
By statement (1), we know that $\langle S_n \rangle$ converges to $S$. 
By statement (2), there is a sequence $\langle \tilde s_n \in p_n^{-1}(s_n) \rangle$
converging to a point $\tilde s \in p^{-1}(s)$. So, by the Projection Theorem,
$s_n=p_n(\tilde s_n)$ converges to $s=p(\tilde s)$. 
\end{proof}
\section{Basepoint change}
\name{sect:basepoint}
Recall that if $(S,s) \in \sE$, then $\BC(S,s)=S^s \in \sM$ is the translation surface which is isomorphic to
the surface $S$ with the basepoint relocated to $s \in S$. We also have basepoint changing isomorphisms
$\beta_s:S \to S^s$. Our goal here is to prove Theorems \ref{thm:basepoint changing map} and \ref{thm:basepoint changing isomorphism}, which claim that these maps are continuous and jointly continuous, respectively.

We will utilize work done in \cite{HooperImmersions1}, which already proved these results for translation structures on disks. We now make a basic observation that describes how we move to the general case.
Let $(S,s) \in \sE$. Let $p:\tilde S \to S$ be the universal covering map, and choose
$\tilde s \in p^{-1}(s)$. Then, we can consider the basepoint changing isomorphisms
$\beta_s:S \to S^s$ and $\tilde \beta_{\tilde s}:\tilde S \to \tilde S^{\tilde s}$. 
We observe that $\BC(\tilde S,\tilde s)=\tilde S^{\tilde s}$ is the universal cover of
$\BC(S,s)=S^s$. Let $p':\tilde S^{\tilde s} \to S^s$ be the universal covering map.
Then, we have the commutative diagram:
\begin{equation}
\name{eq:bc}
\begin{tikzcd}
\tilde S \arrow{r}{\tilde \beta_{\tilde s}} \arrow{d}{p} & \tilde S^{\tilde s} \arrow{d}{p'} \\
S \arrow{r}{\beta_s} & S^s
\end{tikzcd}
\end{equation}

We now prove that the basepoint changing map, $\BC:(S,s) \mapsto S^s$, is continuous.

\begin{proof}[Proof of Theorem \ref{thm:basepoint changing map}]
Let $\langle (S_n,s_n) \in \sE \rangle$ be a sequence converging to $(S,s)$. 
Let $S_n'=\BC(S_n,s_n) \in \sM$ and $S'=\BC(S,s)$. We need to show that $\langle S_n' \rangle$ converges
to $S'$. 

Let $p_n:\tilde S_n \to S_n$, $p_n':\tilde S_n' \to S_n'$, $p:\tilde S \to S$, and $p':\tilde S' \to S'$ be universal coverings.
By Theorem \ref{thm:convergence in E}, there is a sequence of lifts
$\langle \tilde s_n \in p_n^{-1}(s_n) \rangle$ converging to a $\tilde s \in p^{-1}(s)$. 
By Theorem \ref{thm:basepoint disk},
$\tilde S_n'=\BC(\tilde S_n, \tilde s_n) \in \tM$ converges to $\tilde S'=\BC(\tilde S,\tilde s) \in \tM$. 
Let $\tilde \beta_n:\tilde S_n \to \tilde S_n'$ and $\tilde \beta:\tilde S \to \tilde S'$
be the basepoint changing isomorphisms which carry $\tilde s_n$ and $\tilde s$ to the basepoints of $\tilde S_n'$
and $\tilde S'$, respectively. Theorem \ref{thm:basepoint disk} also implies
that if a sequence $\langle \tilde t_n \in \tilde S_n \rangle$ tends to  $\tilde t \in \tilde S$,
then $\langle \tilde \beta_n(\tilde  t_n) \in \tilde S_n' \rangle$ tends to $\tilde \beta(\tilde t) \in \tilde S'$. In addition,
if a sequence $\langle \tilde t_n' \in \tilde S_n' \rangle$ tends to  $\tilde t' \in \tilde S'$,
then $\langle \tilde \beta_n^{-1}(\tilde t_n') \in \tilde S_n \rangle$ tends to $\tilde \beta^{-1}(\tilde t') \in \tilde S$.

To prove $\langle S_n' \rangle$ converges
to $S$, we will utilize the criterion for convergence given by Theorem \ref{thm:convergence in M}.
We begin by verifying condition (1). Let $\tilde t' \in \tilde S'$ be a lift of the basepoint of $S'$. 
Then, $\tilde t=\tilde \beta^{-1}(\tilde t')$ satisfies $p(\tilde t)=s$, where $p: \tilde S \to S$ is the universal covering map. 
By Theorem \ref{thm:convergence in E}, there is a sequence $\langle \tilde t_n \in p_n^{-1}(s_n) \rangle$
which converges to $\tilde t \in \tilde S$. By Theorem \ref{thm:basepoint disk},
the sequence $\langle \tilde t_n'=\tilde \beta_n(\tilde t_n) \rangle$ converges to $\tilde t'=\tilde \beta(\tilde t)$. By the commutative diagram in equation \ref{eq:bc}, each $p_n'(\tilde t_n')$ is the basepoint of $S_n'$. Thus, the fact that 
$\langle \tilde t_n' \rangle$ tends to $\tilde t'$ verifies condition (1).

Now we will check condition (2). Let $\langle n_k \rangle$ be an increasing sequence of integers.
Let $\langle \tilde t_{n_k}' \in \tilde S_{n_k}' \rangle$ be a sequence of lifts of the basepoints of $S_{n_k}'$,
and suppose that the sequence converges to $\tilde t' \in \tE$. Since $\langle \tilde S_{n_k}' \rangle$ converges to $\tilde S'$, we know that $\tilde t' \in \tilde S'$. Let $\tilde t_{n_k}=\tilde \beta_{n_k}^{-1}(\tilde t_{n_k}')$
and $\tilde t=\tilde \beta^{-1}(\tilde t')$. By Theorem \ref{thm:basepoint disk},
$\langle \tilde t_{n_k} \rangle$ tends to $\tilde t$. By the commutative diagram in equation \ref{eq:bc},
each $p_{n_k}(\tilde t_{n_k})=s_{n_k}$. Therefore, the Projection Theorem implies
$p(\tilde t)=s$. Then again by commutativity, $\tilde t'=\tilde \beta(\tilde t)$ is a lift of the basepoint
of $S'$. This verifies condition (2).
\end{proof}

We will now prove the joint continuity of the basepoint changing isomorphism and its inverse.
\begin{proof}[Proof of Theorem \ref{thm:basepoint changing isomorphism}]
Let $\langle (S_n,s_n) \in \sE \rangle$ be a sequence converging to $(S,s) \in \sE$. 
Let $S_n'=\BC(S_n,s_n)$ and $S'=\BC(S,s)$. 
Let $\beta_n:S_n \to S'_n$ and $\beta:S \to S'$ be the associated basepoint changing isomorphisms.

To show that the basepoint changing isomorphism is jointly continuous, it suffices to prove that if
$\langle t_n \in S_n \rangle$ converges to $t \in S$, then $\langle \beta_n(t_n) \rangle$ converges to $\beta(t)$. 

We consider the universal covers using notation from the prior proof.
By Theorem \ref{thm:convergence in E}, there are sequences of lifts
$\langle \tilde s_n \in p_n^{-1}(s_n) \rangle$ converging to a $\tilde s \in p^{-1}(s)$ 
and $\langle \tilde t_n \in p_n^{-1}(s_n) \rangle$ converging to a $\tilde t \in p^{-1}(s)$.
Let $\tilde S_n'=\BC(\tilde S_n, \tilde s_n)$ and $\tilde S'=\BC(\tilde S,\tilde s)$. 
Let $\tilde \beta_n:\tilde S_n \to \tilde S_n'$ and $\tilde \beta:\tilde S' \to \tilde S'$
be the associated basepoint changing isomorphisms. 
By Theorem \ref{thm:basepoint disk}, we know that $\langle \tilde \beta_n(\tilde t_n) \in \tilde S_n' \rangle$ converges to $\tilde \beta(\tilde t)$. 
By the Theorem \ref{thm:basepoint changing map}, we know that $\langle S_n' \rangle$ tends to $S'$
in $\sM$. So we can apply the Projection Theorem
to conclude that $\langle p'_n \circ \tilde \beta_n(\tilde t_n) \in S_n' \rangle$ tends to
$p' \circ \tilde \beta(\tilde t) \in S'$. Then by the commutative diagram in equation \ref{eq:bc}, we see this is the same as
saying that $\langle \beta_n(t_n) \rangle$ converges to $\beta(t)$. 

Theorem \ref{thm:basepoint disk} also states
that the inverse basepoint changing isomorphism is jointly continuous for translation structures on the disk.
With small modifications, the above argument proves that if $\langle t'_n \in S'_n \rangle$ converges to $\langle t' \in S' \rangle$,
then $\langle \beta_n^{-1}(t'_n) \in S_n \rangle$ converges to $\langle \beta^{-1}(t') \in S \rangle$.
\end{proof}
\section{Affine actions}
\name{sect:affine}
The $\GL(2,\R)$ actions behave naturally with respect to universal covering maps. Let $\upsilon:\sM \to \tM$
be the map which sends $S \in \sM$ to its universal cover $\tilde S \in \tM$. The action of any $A \in \GL(2,\R)$ on $\sM$ satisfies the following commutative diagram:
\begin{equation}
\name{eq:A on M}
\begin{tikzcd}[column sep=scriptsize, row sep=scriptsize]
\sM \arrow{r}{A} \arrow{d}{\upsilon} & \sM \arrow{d}{\upsilon} \\
\tM \arrow{r}{A} & \tM
\end{tikzcd}
\end{equation}
Also recall that $\sP$ is the domain of the covering projection
and consists of those pair $(S,\tilde s)$ with $\tilde s \in \tilde S$. The covering projection $p:\sP \to \sE$ sends $(S,\tilde s)$ to its image $p_S(\tilde s)$
under the covering map $p_S:\tilde S \to S$. Because of the diagram above, the $\GL(2,\R)$ actions on $\sM$ and $\sE$ induce an action on $\sP$. We have the following commutative diagram:
\begin{equation}
\name{eq:A on P}
\begin{tikzcd}[column sep=scriptsize, row sep=scriptsize]
\sP \arrow{r}{A} \arrow{d}{p} & \sP \arrow{d}{p} \\
\sE \arrow{r}{A} & \sE
\end{tikzcd}
\end{equation}

\begin{proof}[Proof of Theorem \ref{thm:affine}]
Let $\langle A_n \in \GL(2,\R) \rangle$ be a sequence tending to $A \in \GL(2,\R)$. Let $\langle (S_n,s_n) \in \sE \rangle$ be a sequence
tending to $(S,s) \in \sE$. We will show that $\langle S_n'=A_n(S_n)\rangle$ converges to $S'=A(S)$ in $\sM$, and $\langle s_n'=A_n(S_n,s_n)\rangle$ tends to $s'=A(S,s)$ in $\sE$. 

First we address convergence of $\langle S_n'\rangle$ to $S$. By Theorem \ref*{I-thm:H continuous} of \cite{HooperImmersions1}, we know that the $\GL(2,\R)$
action on $\tM$ is continuous. Then using the continuity of $\upsilon:\sM \to \tM$ and commutativity provided by equation \ref{eq:A on M}, we see
that the universal covers $\langle \tilde S_n' \rangle$ converge to $\tilde S'$ in $\tM$. 

To show $\langle S_n'\rangle$ to $S$, we will use the convergence criterion of Theorem \ref{thm:convergence in M}.
Consider statement (1). We will use $p_n$, $p_n'$, $p$ and $p'$ for universal covering maps associated to $S_n$, $S_n'$, $S$ and $S'$, respectively.
Let $\tilde o' \in p'^{-1}(o_{S'}) \subset \tilde S'$. Let $\tilde o=A^{-1}(\tilde o') \in \tilde S$. Then by the commutative diagram given in equation \ref{eq:A on P}, $\tilde o \in p^{-1}(o_S)$. So by Theorem \ref{thm:convergence in M}, there is a sequence of points $\langle \tilde o_n \in p_n^{-1}(o_{S_n}) \rangle$
converging to $\tilde o$. By Theorem \ref*{I-thm:H continuous} of \cite{HooperImmersions1}, the $\GL(2,\R)$ action on $\tE$ is continuous. Therefore,
$\langle \tilde o_n'=A_n(\tilde o_n) \rangle$ converges to $\tilde o'$. Further, by commutativity we have that
$p_n'(\tilde o_n')$ is the basepoint of $S_n'$. This verifies statement (1).

Now we verify statement (2) of Theorem \ref{thm:convergence in M}. Fix an increasing sequence of positive integers, $\langle n_k \rangle$. We will abuse notation by using $k$ to abbreviate $n_k$. Suppose that there is a sequence $\langle \tilde o_k' \in p_k'^{-1}(o_{S_k'}) \rangle$ which converges to some point $\tilde o' \in \tE$. We must show that $p'(\tilde o')=o_{S'}$. Since $\langle \tilde S_k' \rangle$ converges to $\tilde S'$, we know that $\tilde o' \in \tilde S'$. By continuity of the affine action on $\tE$, we know that $\langle \tilde o_k=A_{k}^{-1}(\tilde o_k')\rangle$ converges to $\tilde o=A^{-1}(\tilde o')$. By commutative diagram \ref{eq:A on P}, we know that $p_k(\tilde o_k)=o_{S_k}$. So, by Theorem \ref{thm:convergence in M} applied to convergence of $\langle S_k \rangle$ to $S$,
we know that $p(\tilde o)=o_S$. Then again by commutativity, we see that $p'(\tilde o')=A(o_S)=o_{S'}$. This verifies statement (2),
and concludes the proof that $\langle S_n'\rangle$ converges to $S$.

It remains to show that $\langle s_n' \rangle$ converges to $s'$. For this we would use Theorem \ref{thm:convergence in E}.
Statement (1) of this theorem has already been proved above. The second statement is proved in an almost identical way to the proof of criterion (1)
of Theorem \ref{thm:convergence in M} given two paragraphs above.
\end{proof}

\section{The compactness theorem}
\name{sect:compactness} 

\begin{proof}[Proof of Theorem \ref{thm:compactness}]
Since the space $\sM$ is second-countable, sequential compactness is equivalent to compactness.
Let $\langle S_n \rangle$ be a sequence in $\sM \smallsetminus \sM_{\not \emb}(U)$. 
We will find a convergent subsequence. We note that it suffices to consider the case when $U$
is isometric to an open Euclidean metric ball with arbitrary radius $\epsilon>0$ 
and basepoint at its center.

Consider the map $\upsilon: \sM \to \tM$ which sends a surface to its universal cover.
This map is continuous and 
$$\upsilon\big(\sM \smallsetminus \sM_{\not \emb}(U)\big) = \tM \smallsetminus \tM_{\not \emb}(U).$$
Note that $\tM \smallsetminus \tM_{\not \imm}(U)$ is compact by Theorem \ref*{I-thm:compactness} of \cite{HooperImmersions1}. 
Since $\tM \smallsetminus \tM_{\not \emb}(U) \subset \tM \smallsetminus \tM_{\not \imm}(U)$,
this set is also compact. Thus, we can extract a convergent subsequence of $\langle \tilde S_n \rangle$
which converges to some surface $\tilde S \in \tM \smallsetminus \tM_{\not \emb}(U)$. We can
assume by passing to such a subsequence that $\langle \tilde S_n \rangle$ converges to $\tilde S$.

We will explicitly find a convergent subsequence, which we will describe as an algorithm. The algorithm can be interpreted as an inductive sequence of definitions. In order to describe the algorithm, we
first pick a sequence of open metric balls of radius less than $\frac{\epsilon}{2}$, $\langle B_i \subset \tilde S \rangle_{i \in \N}$, so that $\tilde S=\bigcup_{i \in \N} B_i$. The purpose of the algorithm
is to find a convergent subsequence $\langle S_{n_k} \rangle$ of $\langle S_n \rangle$. So, we will define
an increasing sequence of natural numbers $\langle n_k \rangle$. We will also construct a
subset $\sI \subset \N$ and a collection $\tilde O=\{\tilde o_i \in \tilde S~:~i \in \sI\}$ indexed with repeats.
The collection $\tilde O \subset \tilde S$ will later be the collection of lifts of the basepoint of the limiting surface $S$. After describing the algorithm, we will prove several statements about the objects produced,
and define $S$. 

We now specify some notation. We let $p_n:\tilde S_n \to S_n$ denote the universal covering maps.
We denote the basepoint of $S_n$ by $o_n$. 

The following is the aforementioned algorithm:
\begin{enumerate}
\item Set $N_0=\N=\{1,2,3, \ldots \}$.
\item Set $\sK=\emptyset \subset \N$.
\item Set $\tilde O =\emptyset \subset \tilde S$.
\item Evaluate the following statements for each integer $k \geq 1$ in order of increasing $k$:
\begin{enumerate}
\item If there is an increasing sequence of integers $\langle m_j=m_j(k) \in N_{k-1}~:~j \in \N\rangle$
and a sequence $\langle \tilde o^k_{m_j} \in p_{m_j}^{-1}(o_{m_j}) \rangle$ which converges to a point
$\tilde o^k \in B_k$, then choose such a sequence and evaluate the following:
\begin{enumerate}
\item Add $\tilde o^k$ to the set $\tilde O \subset \tilde S$. 
\item Add $k$ to the set $\sK \subset \N$. 
\item Define $n_k=m_1(k)=\min \{m_j(k)~:~j \in \N\}$. 
\item Define $N_k=\{m_j(k):j \in \N\} \smallsetminus \{n_k\} \subset N_{k-1}$.
\end{enumerate}
\item Otherwise (if there are no such sequences $\langle m_j \rangle$ and $\langle \tilde o^k_{n_j} \rangle$), evaluate the following statements:
\begin{enumerate}
\item[(v)] Define $n_k=\min N_{k-1}$.
\item[(vi)] Define $N_k=N_{k-1} \smallsetminus \{n_k\}$.
\end{enumerate}
\end{enumerate}
\end{enumerate}
We will show that the subsequence $\langle S_{n_k} \rangle$ converges by using the criterion from Theorem \ref{thm:convergence in M} applied to the constructed subset $\tilde O = \{\tilde o^k:k \in \sK\}$.

We now consider the first statement of Theorem \ref{thm:convergence in M}. Choose a $\tilde o^i \in \tilde O$. 
Then $i \in \sK$, and there is a sequence $\langle m_j(i) \rangle$ and a choice of $\langle \tilde o^i_{m_j(i)} \in p_{m_j}^{-1}(o^i_{m_j(i)}) \rangle$ which converges to $\tilde o^i$. By construction, our constructed sequence $\langle n_k ~:~ k \geq i \rangle$ lies in $N_k$ and thus is a subsequence of $\langle m_{j}(i)~:~j \in \N\rangle$.
In particular, $\langle \tilde o^i_{n_k}~:~k \geq i \rangle$ converges to $\tilde o^i$. This proves that
statement (1) of Theorem \ref{thm:convergence in M} holds.

Now we turn our attention to statement (2). Suppose that $\langle k(\ell):\ell \in \N \rangle$ is an increasing sequence in $\N$, and that there is a sequence $\langle \tilde o_{n_{k(\ell)}} \rangle$ which converges in $\tE$. Let $\tilde o$ be the limit point, which must lie in $\tilde S$. We must show that $\tilde o \in \tilde O$. Since $\{B_i\}$ was a covering of $\tilde S$, there is an $i$ so that $\tilde o \in B_i$. As above,
we note that $\langle n_k:k \geq i\rangle$ is a subsequence of $N_{k-1}$. In particular, there is a subsequence
of $N_{k-1}$ which converges to a point in $B_i$. So, in step (a), we must have chosen a sequence 
$\langle m_j=m_j(i) \in N_{j-1} \rangle$ and a sequence $\langle \tilde o^i_{m_j} \rangle$ which converges
to a $\tilde o^i \in B_i$. We will show that $\tilde o=\tilde o^i$, which will prove that $\tilde o \in \tilde O$. 
Again note that $\langle n_k:k \geq i\rangle$ is a subsequence of $\langle m_j \rangle$. The sequence $\langle n_{k(\ell)}:\ell \in \N\rangle$ is a further subsequence. Therefore, we have 
$$\lim_{\ell \to \infty} \tilde o_{n_{k(\ell)}}=\tilde o \and
\lim_{\ell \to \infty} \tilde o^i_{n_{k(\ell)}}=\tilde o^i.$$
Assume that $\tilde o \neq \tilde o^i$. Choose a $D \in \cdisk(\tilde S)$ so that $B_i \subset D^\circ$. 
Since the limit points are distinct, we can choose disjoint open sets $U \subset B_i$ and $V \subset B_i$
so that $\tilde o \in U$ and $\tilde o^i \in V$. Since $\langle \tilde S_{n_{k(\ell)}} \rangle$ converges to $S$,
there is an $L_1$ so that for $\ell>L_1$, we have an immersion $\iota_\ell:D \imm S_{n_{k(\ell)}}$.
Now observe that $\tilde o \in \sE_+(D,U)$. Therefore, there is an $L_2>L_1$ so that
for $\ell>L_2$, we have $\tilde o_{n_{k(\ell)}}\in \iota_\ell(U)$. Similarly, there is an $L_3>L_2$ so that
for $\ell>L_3$, we have $\tilde o^i_{n_{k(\ell)}}\in \iota_\ell(V)$. In particular, 
$$\tilde o_{n_{k(\ell)}}, \tilde o^i_{n_{k(\ell)}} \in \iota_\ell(B_i) \quad \text{for $\ell>L_3$.}$$
Since $B_i$ is a ball of radius less than $\frac{\epsilon}{2}$, these points are distance less
than $\epsilon$ apart in $\tilde S_{n_{k(\ell)}}$. But, since $U$, a Euclidean open ball of radius $\epsilon$,
embeds into $\tilde S_{n_{k(\ell)}}$ about the basepoint, it must be that these points are equal when $\ell>L_3$. 
Therefore the limits of these sequences are the same, and $\tilde o=\tilde o^i \in \tilde O$.
\end{proof}

\bibliographystyle{amsalpha}
\bibliography{/home/pat/active/my_papers/bibliography}

\newcommand{\cn}[1]{{\bf [#1]:}} \def\scr{\mathcal} \def\cprime{$'$}
\providecommand{\bysame}{\leavevmode\hbox to3em{\hrulefill}\thinspace}
\providecommand{\MR}{\relax\ifhmode\unskip\space\fi MR }
\providecommand{\MRhref}[2]{%
  \href{http://www.ams.org/mathscinet-getitem?mr=#1}{#2}
}
\providecommand{\href}[2]{#2}
\begin{thebibliography}{{Tab}05}

\bibitem[Bow13]{Bowman13}
Joshua~P. Bowman, \emph{The complete family of {A}rnoux-{Y}occoz surfaces},
  Geometriae Dedicata \textbf{164} (2013), no.~1, 113--130 (English).

\bibitem[BV13]{BV13}
Joshua~P. Bowman and Ferr{\'a}n Valdez, \emph{Wild singularities of flat
  surfaces}, Israel J. Math. \textbf{197} (2013), no.~1, 69--97. \MR{3096607}

\bibitem[{Del}13]{Delecroix13}
Vincent {Delecroix}, \emph{{Divergent trajectories in the periodic wind-tree
  model.}}, {J. Mod. Dyn.} \textbf{7} (2013), no.~1, 1--29 (English).

\bibitem[DHL11]{DHL11}
Vincent Delecroix, Pascal Hubert, and Samuel Leli\`evre, \emph{Diffusion for
  the periodic wind-tree model}, preprint \url{http://arxiv.org/abs/1107.1810},
  2011.

\bibitem[FU11]{FUpreprint}
Krzysztof Fr{\c{a}}czek and Corinna Ulcigrai, \emph{Non-ergodic {Z}-periodic
  billiards and infinite translation surfaces}, preprint
  \url{http://arxiv.org/abs/1109.4584}, 2011.

\bibitem[FU12]{FUstrip}
Krzysztof Fraczek and Corinna Ulcigrai, \emph{Ergodic directions for billiards
  in a strip with periodically located obstacles}, 2012.

\bibitem[HHW13]{HHW10}
W.~Patrick Hooper, Pascal Hubert, and Barak Weiss, \emph{Dynamics on the
  infinite staircase}, Discrete and Continuous Dynamical Systems - Series A
  \textbf{33} (2013), no.~9, 4341--4347.

\bibitem[HLT11]{HLT11}
Pascal Hubert, Samuel Leli\`evre, and Serge Troubetzkoy, \emph{{The {E}hrenfest
  wind-tree model: Periodic directions, recurrence, diffusion.}}, J. Reine
  Angew. Math. \textbf{656} (2011), 223--244 (English).

\bibitem[Hoo10]{Hinf}
W.~Patrick Hooper, \emph{The invariant measures of some infinite interval
  exchange maps}, preprint, \url{http://arxiv.org/abs/1005.1902}, 2010.

\bibitem[Hoo13a]{HooperImmersions1}
W.~Patrick Hooper, \emph{Immersions and translation structures on the disk},
  2013, preprint, \href{http://arxiv.org/abs/1309.4795v3}{arXiv:1309.4795v3}.

\bibitem[Hoo13b]{Higl1}
\bysame, \emph{An infinite surface with the lattice property {I}: {V}eech
  groups and coding geodesics}, to appear in Transactions of the AMS,
  \url{http://arxiv.org/abs/1011.0700}, 2013.

\bibitem[HS10]{HS09}
Pascal Hubert and Gabriela Schmith{\"u}sen, \emph{Infinite translation surfaces
  with infinitely generated {V}eech groups}, J. Mod. Dyn. \textbf{4} (2010),
  no.~4, 715--732. \MR{2753950 (2012e:37075)}

\bibitem[HW12]{HW10}
W.~Patrick Hooper and Barak Weiss, \emph{{Generalized staircases: recurrence
  and symmetry.}}, Ann. Inst. Fourier \textbf{62} (2012), no.~4, 1581--1600
  (English. French summary).

\bibitem[HW13]{HW13}
Pascal {Hubert} and Barak {Weiss}, \emph{Ergodicity for infinite periodic
  translation surfaces}, {Compos. Math.} \textbf{149} (2013), no.~8, 1364--1380
  (English).

\bibitem[MT02]{MT}
Howard Masur and Serge Tabachnikov, \emph{Rational billiards and flat
  structures}, Handbook of dynamical systems, Vol.\ 1A, North-Holland,
  Amsterdam, 2002, pp.~1015--1089. \MR{1928530 (2003j:37002)}

\bibitem[PSV11]{PSV11}
Piotr {Przytycki}, Gabriela {Schmith\"usen}, and Ferr\'an {Valdez},
  \emph{{Veech groups of Loch Ness monsters.}}, {Ann. Inst. Fourier}
  \textbf{61} (2011), no.~2, 673--687 (English).

\bibitem[RT12]{RTarxiv12}
David Ralston and Serge Troubetzkoy, \emph{Ergodic infinite group extensions of
  geodesic flows on translation surfaces}, Journal of Modern Dynamics
  \textbf{6} (2012), no.~4, 477--497.

\bibitem[RT13]{RTarxiv11}
\bysame, \emph{Ergodicity of certain cocycles over certain interval exchanges},
  Discrete and Continuous Dynamical Systems \textbf{33} (2013), 2523--2529.

\bibitem[Sch11]{Schmoll11preprint}
Martin Schmoll, \emph{Veech groups for holonomy free torus covers}, Journal of
  Topology and Analysis (2011), to appear.

\bibitem[{Tab}05]{T05}
Serge {Tabachnikov}, \emph{{Geometry and billiards.}}, Providence, RI: American
  Mathematical Society (AMS), 2005 (English).

\bibitem[Thu97]{Thurston}
William~P. Thurston, \emph{Three-dimensional geometry and topology. {V}ol. 1},
  Princeton Mathematical Series, vol.~35, Princeton University Press,
  Princeton, NJ, 1997, Edited by Silvio Levy. \MR{1435975 (97m:57016)}

\bibitem[Tre12]{TrevinoFinite}
Rodrigo Trevi{\~n}o, \emph{On the ergodicity of flat surfaces of finite area},
  2012, preprint, \url{http://arxiv.org/abs/1211.1313}.

\bibitem[Tro10]{Troubetzkoy10}
Serge Troubetzkoy, \emph{{Typical recurrence for the Ehrenfest wind-tree
  model.}}, J. Stat. Phys. \textbf{141} (2010), no.~1, 60--67 (English).

\bibitem[Vee89]{V}
W.~A. Veech, \emph{Teichm\"uller curves in moduli space, {E}isenstein series
  and an application to triangular billiards}, Invent. Math. \textbf{97}
  (1989), no.~3, 553--583. \MR{1005006 (91h:58083a)}

\bibitem[ZK75]{ZK}
A.~N. Zemljakov and A.~B. Katok, \emph{Topological transitivity of billiards in
  polygons}, Mat. Zametki \textbf{18} (1975), no.~2, 291--300. \MR{0399423 (53
  \#3267)}

\bibitem[Zor06]{Zorich06}
Anton Zorich, \emph{Flat surfaces}, Frontiers in number theory, physics, and
  geometry. I, Springer, Berlin, 2006, pp.~437--583. \MR{2261104 (2007i:37070)}

\bibitem[Zvo12]{Zvonkine12}
Dimitri Zvonkine, \emph{An introduction to moduli spaces of curves and their
  intersection theory}, Handbook of {T}eichm\"uller theory. {V}olume {III},
  IRMA Lect. Math. Theor. Phys., vol.~17, Eur. Math. Soc., Z\"urich, 2012,
  pp.~667--716. \MR{2952773}

\end{thebibliography}
\end{document}